\documentclass[hidelinks,onefignum,onetabnum]{siamart250211}

\usepackage{lipsum}
\usepackage{amsfonts}
\usepackage{graphicx}
\usepackage{enumitem}
\usepackage{amssymb}
\usepackage{xcolor}
\usepackage{subfig}
\usepackage{subcaption}
\usepackage{epstopdf}
\usepackage{algorithmic}
\ifpdf
  \DeclareGraphicsExtensions{.eps,.pdf,.png,.jpg}
\else
  \DeclareGraphicsExtensions{.eps}
\fi

\newsiamremark{remark}{Remark}
\newsiamremark{hypothesis}{Hypothesis}
\crefname{hypothesis}{Hypothesis}{Hypotheses}
\newsiamthm{claim}{Claim}
\newsiamremark{fact}{Fact}
\crefname{fact}{Fact}{Facts}

\headers{ItsOPT: An inexact two-level smoothing framework}{A. Kabgani and M. Ahookhosh}

\title{ItsOPT: An inexact two-level smoothing framework for nonconvex optimization via high-order Moreau envelope\thanks{Submitted to the editors DATE.
\funding{The Research Foundation Flanders (FWO) research project G081222N and UA BOF DocPRO4 projects with ID 46929 and 48996 partially supported the paper's authors.}}}

\author{Alireza Kabgani\thanks{Department of Mathematics, University of Antwerp, Antwerp, Belgium.
  (\email{alireza.kabgani@uantwerp.be}, \email{masoud.ahookhosh@uantwerp.be}).}
\and Masoud Ahookhosh\footnotemark[2]}

\usepackage{amsopn}

\ifpdf
\hypersetup{
  pdftitle={ItsOPT: An inexact two-level smoothing framework},
  pdfauthor={Alireza Kabgani and  Masoud Ahookhosh}
}
\fi
\newcommand\R{\mathbb{R}}
\newcommand\Rinf{\overline{\mathbb{R}}}
\newcommand\inter[1]{ {\rm \textbf{int}}(#1)} 
\newcommand\dom[1]{ \bs{{\rm dom}}(#1)} 
\newcommand\Dom[1]{ \bs{{\rm Dom}}(#1)} 
\newcommand\dist{ \bs{{\rm dist}}} 
\newcommand\gf{\varphi} 
\newcommand\gh{\psi} 
\newcommand\fgam[3]{#1_{#3}^{#2}}

\newcommand\fgamepsk[3]{#1^{#2,\varepsilon_k}_{#3}}
\newcommand\fgamepsko[3]{#1^{#2,\varepsilon_{k+1}}_{#3}}
\newcommand\prox[3]{ \bs{{\rm prox}}_{#2#1}^{#3}}

\newcommand\ov[1]{\overline{#1}}
\newcommand\mb{\mathbf{B}}
\newcommand\bs[1]{\boldsymbol{#1}}
\newcommand\argmin[1]{\bs{\arg\min}_{#1}}
\newcommand\argmint[1]{\mathop{\bs{\arg\min}}\limits_{#1}}
\newcommand\Nz{\mathbb{N}_0}
\newcommand\fv{\widehat{\gf}}

\begin{document}

\maketitle

\begin{abstract}
This paper introduces ItsOPT, an {\it inexact two-level smoothing optimization framework} designed 
to find first-order critical points of nonsmooth and nonconvex functions. The framework consists of two levels of methodologies: at the upper level, a zeroth-, first-, or second-order method can be tailored to minimize a smooth approximation; at the lower level, the high-order proximal auxiliary problems are solved inexactly, generating an inexact oracle for the smooth function. As a smoothing technique, we introduce the high-order Moreau envelope (HOME)
and study its fundamental properties under standard assumptions. Next, by combining a boosted high-order proximal-point algorithm (Boosted HiPPA) at the upper level with the inexact oracle from the lower level, we obtain
a zeroth-order instance of ItsOPT. Global convergence rates are established under the Kurdyka-{\L}ojasiewicz (KL) property of the cost and envelope functions, together with reasonable conditions on the accuracy of the proximal terms. Surprisingly, for any KL exponent $\theta\in (0,1)$ of the original cost, setting the regularization order $p=\frac{1}{1-\theta}$ ensures that Boosted HiPPA converges linearly to a proximal fixed point. This is the first algorithm with this property for KL functions. Preliminary numerical experiments on a robust low-rank matrix recovery problem demonstrate the promising performance of the proposed algorithm, supporting our theoretical foundations. 
\end{abstract}



\section{Introduction}\label{intro}

This paper deals with nonconvex optimization problems 
\begin{equation}\label{eq:mainproblem}
\mathop{\bs\min}\limits_{x\in \mathbb {R}^n}\  \gf(x),
\end{equation}
subject to the following basic assumptions.
\begin{assumption}[Basic assumptions]\label{ass:basic} In problem \eqref{eq:mainproblem}, we assume that
\begin{enumerate}[label=(\textbf{\alph*}), font=\normalfont\bfseries, leftmargin=0.7cm]
\item  $\gf: \R^n\to\Rinf:=\R\cup \{+\infty\}$ is a proper, lower semicontinuous (lsc), high-order prox-bounded function with a threshold $\gamma^{\gf, p}>0$ (see Definition~\ref{def:s-prox-bounded}), and is possibly nonsmooth;
\item the set of minimizers is nonempty, with the corresponding minimal value $\gf^*$.
\end{enumerate}
\end{assumption}
In the above problem, we emphasize that no special structure is imposed on the underlying function $\gf$, allowing it to encompass a wide range of optimization problems, including those where the objective function $\gf$ combines functions via summation, distance, fraction, or composition.

Solving nonsmooth optimization problems by tailoring iterative schemes based on smoothing techniques has a long history in optimization; see, e.g., \cite{Beck12, ben2006smoothing, bertsekas2009nondifferentiable, Bot2020, Bot15, Moreau65, nesterov2005smooth}, intending to develop methods that outperform classical approaches such as subgradient-based and bundle methods. The main idea is to replace the nonsmooth cost function with a smoothing counterpart that has the same set of minimizers, and then minimize this surrogate by smooth optimization approaches, which may consequently lead to faster schemes than classical nonsmooth methods. In particular, several smoothing techniques (in both Euclidean and non-Euclidean settings) have been studied for composite minimization; see, e.g., \cite{Ahookhosh21, patrinos2013proximal, Stella17, themelis2019acceleration,themelis2020douglas,Themelis18}.

Among all developments, the {\it Moreau envelope} has received significant attention due to its simplicity and elegant properties, dating back to the seminal work by Jean-Jacques Moreau in 1965 \cite{Moreau65}. 
The {\it proximal-point operator} and the {\it Moreau envelope} are defined as
\begin{equation*}\label{eq:prox2}
\prox{\gf}{\gamma}{} (x):=\argmint{y\in \R^n} \left(\gf(y)+\frac{1}{2\gamma}\Vert x- y\Vert^2\right),
    \fgam{\gf}{}{\gamma}(x):=\mathop{\bs{\inf}}\limits_{y\in \R^n} \left(\gf(y)+\frac{1}{2\gamma}\Vert x- y\Vert^2\right),
\end{equation*}
where the proximal operator has been used to generate the proximal-point algorithm (i.e., $x^{k+1}\in\mathop{\bs{\mathrm{prox}}}_{\gamma \varphi}(x^k)$); see, e.g., \cite{martinet1970breve, martinet1972determination, rockafellar1976monotone}.
If $\varphi$ does not have a complicated structure, the proximal-point operator admits a closed-form solution \cite{beck2017first}, which leads to efficient algorithms due to its {\it simple structure} and {\it low-memory requirement}; cf. \cite{Parikh14}. Moreover, the Moreau envelope shares the same minimizers as the original cost function $\varphi$ and enjoys favorable {\it differential properties} under {\it convexity} or {\it prox-regularity}; see, e.g., \cite{Bauschke17, Kabganitechadaptive, Kabganidiff, KecisThibault15, Poliquin96, Rockafellar09}. Consequently, both the proximal-point operator and the Moreau envelope have been standing as unavoidable tools for solving nonsmooth optimization problems.

Although extensive research has been conducted on proximal-point methods, most studies rely on the existence of a {\it closed-form solution} for the proximal-point operator, which holds only for a limited class of simple functions; see, e.g., \cite[Chapter~6]{beck2017first}. Such approaches fail 
in cases where $\varphi$ has a more complex structure. This has motivated the development of {\it inexact proximal-point methods}; see, e.g., 
\cite{Ahookhosh24,barre2023principled,Dvurechenskii2022,liu2012implementable,Nesterov2023a,rockafellar1976monotone,Salzo12,Solodov01} and references therein. 
In particular, Nesterov proposed a {\it bi-level unconstrained convex minimization (BLUM)} framework, introducing a novel high-order proximal-point operator using a $p$th-order regularization term
\begin{equation}\label{eq:proxp1}
    \bs{\mathrm{prox}}_{\varphi,H}^{p}(x):= \argmint{y\in \R^n} \left(\varphi(y)+\frac{H}{p} \|x-y\|^{p}\right),
\end{equation}
for a smooth function $\varphi$, $H> 0$, and $p\geq 2$; cf. \cite{nesterov2021inexact,Nesterov2023a}. This framework consists of two levels, where the upper level applies the inexact high-order proximal-point operator, and the lower level solves the underlying proximal problem inexactly. This structure has led to interesting complexity results; see, e.g., \cite{Ahookhosh24,Nesterov2023a,Nesterov2022,zhu2024global}. 
Recently, Ahookhosh and Nesterov extended this framework further to convex composite optimization in a {\it bi-level optimization framework} (BiOPT), where the related proximal auxiliary problem is solved inexactly by the Bregman proximal gradient method; cf. \cite{Ahookhosh23,Ahookhosh24}.

As a consequence of the definition for the Moreau envelope and other envelopes for composite problems (e.g., \cite{Ahookhosh21, patrinos2013proximal, Stella17,  themelis2019acceleration,themelis2020douglas,Themelis18}), an {\it exact oracle} for these functions exists only if a closed-form solution to the underlying proximal auxiliary problem is available. Otherwise, such an oracle cannot be generated. In practice, this requirement is often the most serious bottleneck in applying envelope functions, rendering them impractical for many interesting applications with complex cost functions. 
To the best of our knowledge, no study has addressed inexactness of proximal oracles for envelope functions. This motivates the quest to develop a {\it novel smoothing framework} based on envelope functions with {\it inexact proximal oracles}.  
Accordingly, this leads to an {\it inexact two-level smoothing optimization framework}, where at the upper level, a smooth optimization scheme with an inexact oracle is applied, and at the lower level, the corresponding proximal auxiliary problem is solved inexactly. The development of such a framework is the main motivation for our study, and its details will be described in the coming sections.

\subsection{Contribution and related works}
Our contributions are threefold:
\begin{description}[wide, labelwidth=!, labelindent=0pt]
    \item[(i)] {\bf Inexact Two-level Smoothing OPTimization (ItsOPT) framework.} 
     We introduce the ItsOPT framework, an inexact two-level smoothing approach for nonsmooth and nonconvex problems. It extends BLUM \cite{nesterov2021inexact} (for smooth convex problems) and BiOPT \cite{Ahookhosh24} (for nonsmooth convex problems) to a broader class. This framework employs a proximal envelope to smooth the nonsmooth cost function and consists of {\it two levels of methodologies}: (i) {\it upper-level methods} may be gradient-based, quasi-Newton, or generalized second-order (using the second-order generalized derivative tools \cite{mordukhovich2024second}) scheme; (ii) {\it the lower-level method} approximately solves the underlying high-order proximal auxiliary problem~\eqref{eq:proxp1}.
     
    \item[(ii)] {\bf Fundamental properties of HOPE and HOME.} 
  We define the {\it high-order proximal operator (HOPE)} and the {\it high-order Moreau envelope (HOME)} for nonsmooth and nonconvex problems, and establish their key properties.
Due to our optimization purpose, we characterize the {\it relationships among various first-order reference points}  of HOME and the original cost function. To study the convergence rate of algorithms generated to minimize HOME, we investigate the relationship between the {\it KL property for HOME} and the original cost function.  
Let us emphasize that {\it HOME does not need to know about the structure of the problem} as opposed to the (Bregman) forward-backward \cite{Ahookhosh21,patrinos2013proximal,Stella17,Themelis18}, Douglas-Rachford \cite{themelis2020douglas}, Davis-Yin \cite{liu2019envelope} envelopes that require Lipschitz smoothness (or relative smoothness) of the differentiable part of the cost function, which is a considerable advantage of HOME over these envelopes for many practical applications.

    \item[(iii)] {\bf Boosted HiPPA as a special case of ItsOPT.}
    As a first natural instance of ItsOPT, we introduce the {\it inexact high-order proximal-point algorithm (HiPPA)}, where the upper-level follows a generic proximal-point scheme with $p$th-order regularization ($p>1$) and the lower-level solves the proximal auxiliary problem inexactly. 
To further accelerate this algorithm, we propose the {\it boosted version (Boosted HiPPA)} and its adaptive counterpart by adopting a {\it derivative-free nonmonotone line search} in a similar vein as the schemes given in  \cite{fukushima1981generalized,fukushima1996,mine1981minimization}, where the {\it search direction is not necessarily a descent direction}. 
Since only approximate proximal operators are available, we substitute them into HOME to obtain an inexact zeroth-order oracle (inexact function value). By introducing {\it suitable inexactness conditions}, we design stopping criteria for the lower-level method to solve the proximal auxiliary problem.  We prove the well-definedness and subsequential convergence of Boosted HiPPA under standard assumptions, and establish global and local linear convergence under the KL property. {\it Surprisingly, if the original cost is a KL function with exponent $\theta \in (0, 1)$, by setting $p = \frac{1}{1-\theta}$, Boosted HiPPA converges linearly. To the best of our knowledge, Boosted HiPPA is the first algorithm achieving linear convergence for any KL function by leveraging a proper proximal regularization}. Finally, we report promising numerical results for robust low-rank matrix recovery problems.
\end{description}

\subsection{Paper organization}
In Section~\ref{sec:preliminaries}, we introduce the necessary notation and preliminaries.
Section~\ref{sec:itsopt} introduces the ItsOPT framework and deals with the fundamental properties of HOPE and HOME. 
Section~\ref{sec:boosted} describes HiPPA and Boosted HiPPA together with their convergence analysis. Section~\ref{sec:numerical} demonstrates the application of our algorithms to robust low-rank matrix recovery. We summarize, extend, and discuss the implications and limitations of our results in Section~\ref{sec:disc}.

\section{Preliminaries and notation} \label{sec:preliminaries}

In this paper, $\R^n$ denotes an $n$-dimensional \textit{Euclidean space} equipped with the \textit{Euclidean norm} $\Vert\cdot\Vert$, and $\langle\cdot, \cdot\rangle$ represents the standard \textit{inner product}. Additionally, $\mathbb{N}$ is the set of \textit{natural numbers}, and we set $\Nz:=\mathbb{N}\cup\{0\}$.
We denote the \textit{open ball} with center $\ov{x}\in \R^n$ and radius $r>0$ as $\mb(\ov{x}; r)$.
The \textit{interior} of a set $C\subseteq \R^n$ is denoted by $\inter{C}$.
The \textit{distance} from $x\in \R^n$ to a nonempty set $C\subseteq\R^n$ is defined as 
$\dist(x,C):=\bs\inf_{y\in C}\Vert y - x\Vert$.
For a sequence $\{x^k\}_{k\in \Nz}$, $\ov{x}$ is a \textit{limiting point} if $x^k\to \ov{x}$, and $\widehat{x}$ is a \textit{cluster point} of this sequence if there exist an infinite subset $J\subseteq \Nz$ and a subsequence $\{x^{j}\}_{j\in J}$ such that $x^j\to \widehat{x}$.

For  $\gh: \R^n \to \Rinf:=\R\cup\{+\infty\}$, 
the set $\dom{\gh}:= \{x \in   \R^{n}\mid~\gh(x)< + \infty \}$ denotes
the \textit{effective domain} of $\gh$ and it
is said to be \textit{proper} if $\dom{\gh}\neq \emptyset$.
The set of all minimizers of $\gh$ on $C\subseteq\R^n$ is denoted by $\argmin{x\in C}\gh(x)$.
The function $\gh$ is \textit{lower semicontinuous} (lsc) at $\ov{x} \in \R^{n}$ if, for any sequence $\{x^k\}_{k\in \Nz} \subseteq \R^{n}$ with $x^k \rightarrow \ov{x}$, we have $\bs\liminf_{k\rightarrow + \infty} \gh(x^k)\geq \gh(\ov{x})$. 
The function $\gh$ is referred to as lsc if it has this property at each $x \in \R^n$.
The domain of a set-valued mapping $T: \R^n \rightrightarrows \R^n$ is defined as $\Dom{T}:=\{x\in\R^n\mid T(x)\neq\emptyset\}$.
For $p > 1$, the gradient of $\frac{1}{p} \Vert x \Vert^p$ is given by $\nabla \left(\frac{1}{p} \Vert x \Vert^p \right) = \Vert x \Vert^{p-2} x$, adopting the convention $\frac{0}{0} = 0$ when $x = 0$.

The following inequalities will be frequently used in the coming sections.
\begin{lemma}[Basic inequalities]\label{lemma:ineq:inequality p} 
Assume that $a, b\in \R^n$.
\begin{enumerate}[label=(\textbf{\alph*}), font=\normalfont\bfseries, leftmargin=0.7cm]
\item \label{lemma:ineq:inequality p:ineq1} For each $\lambda\in (0,1)$ and $p\geq 1$, 
$\Vert a+b\Vert^p\geq \lambda^{p-1}\Vert a\Vert^p-\left(\frac{\lambda}{1-\lambda}\right)^{p-1}\Vert b\Vert^p$.
\item \label{lemma:ineq:inequality p:ineq4} For each $p\geq 1$, $\Vert a-b\Vert^p\leq 2^{p-1}\left(\Vert a\Vert^p+\Vert b\Vert^p\right)$.
\end{enumerate}
\end{lemma}
\begin{proof}
\ref{lemma:ineq:inequality p:ineq1} The convexity of $\Vert x\Vert^p$ implies that for any $\lambda\in (0, 1)$,
\begin{equation*}
\Vert a\Vert^p=\left\Vert \lambda \frac{a+b}{\lambda}+(1-\lambda)\frac{-b}{1-\lambda}\right\Vert^p\leq \lambda^{1-p}\Vert  a+b\Vert^p+(1-\lambda)^{1-p}\Vert b\Vert^p.
\end{equation*}
\ref{lemma:ineq:inequality p:ineq4} This follows by setting $\lambda=\frac{1}{2}$ and substituting $a-b$ for $a$ in Assertion~\ref{lemma:ineq:inequality p:ineq1}. 
\end{proof}
For two positive real numbers $a$ and $b$,  and for $p \in [0,1]$, it holds that
\begin{equation}\label{eq:intrp:p01}
(a+b)^p\leq a^p+b^p.
\end{equation}

 A proper function $\gh: \R^n \to \Rinf$ is said to be \textit{Fr\'{e}chet differentiable} at a point
 $\ov{x}\in\inter{\dom{\gh}}$ with \textit{Fr\'{e}chet derivative}  
$\nabla \gh(\ov{x})$
 if 
\begin{equation*}
\mathop{\bs\lim}\limits_{x\to \ov{x}}\frac{\gh(x) -\gh(\ov{x}) - \langle \nabla \gh(\ov{x}) , x - \ov{x}\rangle}{\Vert x - \ov{x}\Vert}=0.
\end{equation*}
The \textit{Fr\'{e}chet/regular} and \textit{Mordukhovich/limiting subdifferentials} of $\gh$ at $\ov{x}\in \dom{\gh}$ \cite{Mordukhovich2018,Rockafellar09} are given, respectively, by
\[
\widehat{\partial}\gh(\ov{x}):=\left\{\zeta\in \R^n\mid~\mathop{\bs\liminf}\limits_{x\to \ov{x}}\frac{\gh(x)- \gh(\ov{x}) - \langle \zeta, x - \ov{x}\rangle}{\Vert x - \ov{x}\Vert}\geq 0\right\},
\]
and
\begin{equation*}
\partial \gh(\ov{x}):=\left\{\zeta\in \R^n\mid~\exists x^k\to \ov{x}, \zeta^k\in \widehat{\partial}\gh(x^k),~~\text{with}~~\gh(x^k)\to \gh(\ov{x})~\text{and}~ \zeta^k\to \zeta\right\}.
\end{equation*}

The Kurdyka-\L{}ojasiewicz property is crucial for studying global and linear convergence
of many algorithms; see, e.g., \cite{absil2005convergence,Ahookhosh21,attouch2010proximal,Attouch2013,Bolte2007Clarke,bolte2007lojasiewicz,Bolte2014,Li18,li2023convergence,Yu2022}.
\begin{definition}[Kurdyka-\L{}ojasiewicz property]\label{def:kldef}
Let $\gh: \R^n \to \Rinf$ be a proper lsc function. Then $\gh$ satisfies the \textit{Kurdyka-\L{}ojasiewicz (KL) property} at $\ov{x}\in\Dom{\partial \gh}$ if there exist constants $r>0$, $\eta\in (0, +\infty]$, and a desingularizing function $\phi$ such that
\begin{equation}\label{eq:intro:KLabs}
\phi'\left(\vert\gh(x) - \gh(\ov{x})\vert\right)\dist\left(0, \partial \gh(x)\right)\geq 1,
\end{equation}
whenever $x\in \mb(\ov{x}; r)\cap\Dom{\partial \gh}$ with $0<\vert \gh(x)-\gh(\ov{x})\vert<\eta$, where the function
\linebreak$\phi:[0, \eta)\to[0, +\infty)$ is a concave, continuous function satisfying $\phi(0)=0$, is continuously differentiable on $(0, \eta)$, and satisfies $\phi'>0$ on $(0, \eta)$.
\\
We further say that $\gh$ satisfies the KL property at $\ov{x}$ with an \textit{exponent} $\theta$ if $\phi(t) = c t^{1 - \theta}$ for some $\theta\in [0, 1)$ and a constant $c>0$.
\end{definition}
\begin{remark}\label{rem:KLwabsval}
In the classical KL property setting  \cite{attouch2010proximal}, the inequality \eqref{eq:intro:KLabs} is stated without the absolute value, under the condition $\gh(\ov{x})<\gh(x)<\gh(\ov{x})+\eta$
whenever 
\linebreak$x\in \mb(\ov{x}; r)\cap \Dom{\partial \gh}$. 
This condition is often satisfied near a critical point for monotone algorithms but may fail for nonmonotone ones, making Definition~\ref{def:kldef} more suitable. Recently, Li et al. \cite{li2023convergence} discussed the advantages of this definition for continuously differentiable functions, noting that semialgebraic and subanalytic functions satisfy this property; see also \cite[Corollary~12]{Bolte2007Clarke}.
\end{remark}
We say that $\gh$ satisfies the KL property at $\ov{x} \in \Dom{\partial \gh}$ with the {\it quasi-additivity property}
if, for some constant $c_{\phi} > 0$, the desingularizing function $\phi$ satisfies 
\begin{equation*}
\left[\phi'(x+y)\right]^{-1}\leq c_\phi \left[\left(\phi'(x)\right)^{-1}+\left(\phi'(y)\right)^{-1}\right], \qquad\forall x, y\in (0, \eta)~\text{with}~ x+y<\eta.
\end{equation*}
If $\gh$ satisfies the KL property at $\ov{x}$ with an exponent $\theta$, then $\gh$ also satisfies the KL property with the quasi-additivity property, with $c_{\phi} = 1$ \cite[Page 1098]{li2023convergence}.

In the following, we state the uniformized KL property.
\begin{lemma}[Uniformized KL property]\label{lem:Uniformized KL property}
Let $\gh:\R^n\to \Rinf$ be a proper lsc function that is constant on a nonempty compact set $C\subseteq \Dom{\partial \gh}$.
If $\gh$ satisfies the KL property with the quasi-additivity property at each point of $C$, then there exist $r>0$, $\eta>0$, and a desingularizing function $\phi$, which satisfies the quasi-additivity property, such that for all $\ov{x}\in C$ and $x\in X$, where
\begin{equation*}
X:=\Dom{\partial \gh}\cap\left\{x\in \R^n \mid \dist(x, C)<r\right\}\cap\left\{x\in \R^n \mid 0<\vert \gh(x) - \gh(\ov{x})\vert<\eta\right\},
\end{equation*}
we have
$
\phi'\left(\vert \gh(x) - \gh(\ov{x})\vert\right)\dist\left(0, \partial \gh(x)\right)\geq 1.
$
\end{lemma}
\begin{proof}
The proof follows the same approach as the smooth version of this theorem \cite[Lemma~3.5]{li2023convergence}. We provide a brief sketch. Let $\gh(x)=c$  on $C$ for some constant $c$. Assume that $\{\ov{x}^i\}_{i\in M:=\{1, \ldots, m\}}\subseteq C$ and radii $r_i>0$, $i\in M$, are such that 
$\gh$ has the KL property on $\mb(\ov{x}^i; r_i)$ and $C\subseteq \bigcup_{i\in M}\mb(\ov{x}^i; r_i)$, which is possible due to the compactness of $C$. For each $i\in M$, there exists $\eta_i\in(0, +\infty]$ and a desingularizing function $\phi_i:[0, \eta_i)\to [0, +\infty)$ with the properties given in 
Definition~\ref{def:kldef}, satisfying \eqref{eq:intro:KLabs}. That is, whenever $x\in \mb(\ov{x}^i; r_i)\cap\Dom{\partial \gh}$ with $0<\vert \gh(x)-\gh(\ov{x}^i)\vert<\eta_i$, we have
$\phi'_i\left(\vert \gh(x) -c\vert\right)\dist\left(0, \partial \gh(x)\right)\geq 1$.
By choosing $r > 0$ sufficiently small so that
$
U_r := \left\{ x \in \mathbb{R}^n : \dist(x, C) \leq r \right\} \subseteq \bigcup_{i\in M}\mb(\ov{x}^i; r_i),
$
and setting $\eta =\bs\min_{i\in M}\eta_i$ and $\phi(t) :=\sum_{i\in M}\phi_i(t)$ on $[0, \eta)$,  we obtain for all
$x \in X$,
\begin{equation*}
\phi'(\vert \gh(x) - c\vert) \dist(0, \partial \gh(x)) = \sum_{i\in M} \phi_i'(\vert \gh(x) - c\vert) \dist(0, \partial \gh(x)) \geq 1.
\end{equation*}
The argument that $\phi$ satisfies the quasi-additivity property is identical to the second step of the proof of \cite[Lemma~3.5]{li2023convergence}.
\end{proof}



\section{ItsOPT: inexact two-level smoothing framework} \label{sec:itsopt}
Let us consider the minimization problem \eqref{eq:mainproblem} with cost function $\gf$.
Smoothing methods via envelope functions $\gf_\gamma$ require a closed-form solution to the proximal auxiliary problem, which may not exist for complex cost functions or regularization terms, thereby limiting their applicability. Alternatively, inexact solutions can be computed under specific inexactness criteria, although they cannot directly generate zeroth-, first-, or second-order oracles for $\gf_\gamma$. However, they can provide an inexact oracle $\mathcal{O}(\varphi_\gamma, x)$.

Based on the above discussion, the design of a generic algorithm for \eqref{eq:mainproblem} involves two main steps:
(i) computing an inexact solution of the proximal auxiliary problem to determine an inexact oracle for the envelope $\varphi_\gamma$; 
(ii) developing a zeroth-, first-, or generalized second-order method that depends on the computed inexact oracle. 
This leads to an inexact two-level smoothing framework (ItsOPT), where the upper level employs 
zeroth-order (e.g., proximal point method with line search),
first-order (e.g., gradient, spectral gradient, conjugate gradient), or second-order (e.g., quasi-Newton, trust-region) methods, while the lower level inexactly solves the proximal auxiliary problem using low-cost methods (e.g., subgradient \cite{Davis2018, Rahimi2024}, Bregman gradient \cite{Ahookhosh24}, BELLA \cite{Ahookhosh21}). An instance of this framework is summarized in the following algorithm.

\begin{algorithm}
\caption{ItsOPT (Inexact two-level smoothing OPTimization algorithm)}\label{alg:itsopt}
\begin{algorithmic}[1]
\STATE\label{alg:itsopt:env}\textbf{Smoothing technique} Choose an envelope to generate a smooth variant of $\gf$;
\STATE \textbf{Initialization} Start with $x^0\in \R^n$, $\gamma>0$, and set $k=0$;
\WHILE{stopping criteria are not satisfied}
\STATE Find an approximated solution of the proximal auxiliary problem; \COMMENT{L-level}
\STATE\label{alg:itsopt:orac} Generate the inexact 
zeroth-, first- or second-order oracle $\mathcal{O}(\varphi_\gamma,x^k)$; \COMMENT{L-level}
\STATE\label{alg:itsopt:dir}Choose a direction $d^k\in \R^n$ using the inexact oracle $\mathcal{O}(\varphi_\gamma,x^k)$; \COMMENT{U-level}
\STATE Determine the step-size $\alpha_k>0$ and update $x^{k+1}:=x^k+\alpha_k d^k$; \COMMENT{U-level}
\STATE $k=k+1$;
\ENDWHILE
\end{algorithmic}
\end{algorithm}
In Step~\ref{alg:itsopt:env} of Algorithm~\ref{alg:itsopt}, any proximal envelope can be applied, such as the Moreau envelope \cite{Moreau65}, the high-order Moreau envelope (see Definition~\ref{def:Hiorder-Moreau env}), the forward-backward envelope \cite{patrinos2013proximal,Stella17,Themelis18}, the Bregman Moreau envelope \cite{Ahookhosh21}, the Bregman forward-backward envelope \cite{Ahookhosh21}, the Douglas-Rachford envelope \cite{themelis2020douglas}, the Davis-Yin envelope \cite{liu2019envelope} and their variants. Furthermore, various approaches exist for defining inexactness within this framework; see, e.g., \cite{Ahookhosh24, nesterov2021inexact,rockafellar1976monotone, Salzo12} and the references therein. 

Moreover, the direction $d^k$ in Step~\ref{alg:itsopt:dir} of Algorithm~\ref{alg:itsopt} depends on the oracle chosen in Step~\ref{alg:itsopt:orac}, which may lead to either a zeroth-, first- or generalized second-order method at the upper level of the algorithm. In this paper, we focus on the high-order Moreau envelope as a smoothing technique, combined with a generic proximal-point method, to present the first zeroth-order natural algorithm within the ItsOPT framework. This algorithm will be described in detail in Section~\ref{sec:boosted}. For a first-order approach, see \cite{Kabganitechadaptive}.

In the remainder of this section, we introduce the high-order proximal operator and the high-order Moreau envelope, alongside their key properties that are vital for developing the proposed methods. We also present a framework for constructing an inexact oracle for the high-order Moreau envelope.

\subsection{High-order Moreau envelope in the nonconvex setting}
\label{subsec:home}

Let us begin with the definitions of HOPE and HOME by considering a general power  $p>1$.
\begin{definition}[High-order proximal operator and Moreau envelope]\label{def:Hiorder-Moreau env}
Let $p>1$  and $\gamma>0$, and let $\gf: \R^n \to \Rinf$ be a proper function. 
    The \textit{high-order proximal operator} (\textit{HOPE}) of $\gf$,
    $\prox{\gf}{\gamma}{p}: \R^n \rightrightarrows \R^n$, is defined by
   \begin{equation}\label{eq:Hiorder-Moreau prox}
       \prox{\gf}{\gamma}{p} (x):=\argmint{y\in \R^n} \left(\gf(y)+\frac{1}{p\gamma}\Vert x- y\Vert^p\right).
    \end{equation}     
  The \textit{high-order Moreau envelope} (\textit{HOME}) of $\gf$, 
    $\fgam{\gf}{p}{\gamma}:\R^n\to \R\cup\{\pm \infty\}$, 
    is defined by
    \begin{equation}\label{eq:Hiorder-Moreau env}
    \fgam{\gf}{p}{\gamma}(x):=\mathop{\bs{\inf}}\limits_{y\in \R^n} \left(\gf(y)+\frac{1}{p\gamma}\Vert x- y\Vert^p\right).
    \end{equation}
\end{definition}

\begin{fact}[Domain and majorizer for HOME]\cite[Proposition~12.9]{Bauschke17}\label{fact:horder:Bauschke17:p12.9}
Let $p> 1$ and $\gf:\R^n \to \Rinf$ be a proper function. Then, for every $\gamma > 0$,
$\fgam{\gf}{p}{\gamma}(x)<\infty$ for all $x\in \R^n$.
Additionally, for all $\gamma_2>\gamma_1>0$ and $x\in \R^n$, $\fgam{\gf}{p}{\gamma_2}(x)\leq \fgam{\gf}{p}{\gamma_1}(x)\leq \gf(x)$.
\end{fact}
From Fact~\ref{fact:horder:Bauschke17:p12.9}, $\fgam{\gf}{p}{\gamma}(x)<\infty$ for all $x\in \R^n$. 
To avoid cases where $\fgam{\gf}{p}{\gamma}(x) = -\infty$ for each $\gamma>0$ and $x \in \R^n$, we generalize the prox-boundedness (\cite[Definition~1.23]{Rockafellar09})
 to the case $p > 1$.

\begin{definition}[High-order prox-boundedness]\label{def:s-prox-bounded}
A function $\gf:\R^n\to \Rinf$ is said to be \textit{high-order prox-bounded} with order $p$, for a given $p> 1$, if there exist $\gamma>0$ and $x\in \R^n$ such that
$\fgam{\gf}{p}{\gamma}(x)>-\infty$. 
The supremum of the set of all such $\gamma$ is denoted by $\gamma^{\gf, p}$ and is referred to as the \textit{threshold of high-order prox-boundedness} for $\gf$.
\end{definition}

Lower-bounded and also convex functions satisfy Definition~\ref{def:s-prox-bounded} with $\gamma^{\gf, p} = +\infty$.

We now investigate conditions ensuring the nonemptiness and outer semicontinuity of $\prox{\gf}{\gamma}{p}$, as well as the continuity and finiteness of $\fgam{\gf}{p}{\gamma}$. In the following theorem, \textit{uniform level boundedness} plays a key role (see \cite[Definition~1.16]{Rockafellar09}).
Similar results for $p=2$ can be found in \cite[Theorem~1.25]{Rockafellar09}.

 \begin{theorem}[Basic properties of HOME and HOPE]\label{th:level-bound+locally uniform}
Let $p>1$ and let 
\linebreak$\gf: \R^n\to\Rinf$ be a proper lsc function that is high-order prox-bounded with threshold $\gamma^{\gf, p}$. Then, for each $\gamma\in (0, \gamma^{\gf, p})$, the following statements hold:
\begin{enumerate}[label=(\textbf{\alph*}), font=\normalfont\bfseries, leftmargin=0.7cm]
\item \label{level-bound+locally uniform:levelunif} the function $\Phi(y, x):=\gf(y)+\frac{1}{p\gamma}\Vert x- y\Vert^p$  is
 level-bounded in $y$ locally uniformly in $x$, and lsc in $(y, x)$;
\item \label{level-bound+locally uniform:proxnonemp} $\prox{\gf}{\gamma}{p}(x)$ is nonempty and compact, and $\fgam{\gf}{p}{\gamma}(x)$ is finite for every $x\in \R^n$;
\item \label{level-bound+locally uniform:cononx} $\fgam{\gf}{p}{\gamma}$ is continuous on $\R^n$;

\item \label{level-bound+locally uniform2:conv} if $y^k\in \prox{\gf}{\gamma_k}{p}(x^k)$ with $x^k\to \ov{x}$ and $\gamma_k\to \gamma\in (0, \gamma^{\gf, p})$, then the sequence $\{y^k\}_{k\in \mathbb{N}}$ is bounded and all cluster points of this sequence lie in $\prox{\gf}{\gamma}{p}(\ov{x})$.
     \end{enumerate}
 \end{theorem} 
 \begin{proof}
From high-order prox-boundedness, there exist $\gamma$, $\ov{x}$, and $\ell_0\in \R$ such that
 \begin{equation*}
  \gf(y)+\frac{2^{p-1}}{p\gamma}\left(\Vert \ov{x}\Vert^p+\Vert y\Vert^p\right)\geq  \gf(y)+\frac{1}{p\gamma}\Vert \ov{x}- y\Vert^p\geq \ell_0, \quad \forall y\in \R^n.
  \end{equation*}
Thus, $\gf(y)+\frac{2^{p-1}}{p\gamma}\Vert y\Vert^p\geq \ell_0 -\frac{2^{p-1}}{p\gamma}\Vert \ov{x}\Vert^p$, for all $y\in \R^n$. Leaving $\ell:=\frac{2^{p-1}}{p\gamma}$, this means that $\gf(\cdot)+\ell\Vert \cdot\Vert^p$ is bounded from below on $\R^n$.
We now prove each assertion.
\\
$\ref{level-bound+locally uniform:levelunif}$ 
Suppose, for contradiction, that there exists $\gamma\in (0, \gamma^{\gf, p})$ such that $\Phi(y, x)$ is not
 level-bounded in $y$ locally uniformly in $x$. Then there exist $\ov{x}\in \R^n$, $\ell_1\in \R$, sequences $x^k\to \ov{x}$, and $\{y^k\}_{k\in \mathbb{N}}$ with $\Vert y^k\Vert\to \infty$ such that, for each $k\in \mathbb{N}$,
\begin{equation}\label{eq1:level-bound+locally uniform}
\Phi(y^k, x^k)=\gf(y^k)+\frac{1}{p\gamma}\Vert x^k- y^k\Vert^p\leq \ell_1.
\end{equation}
For any $\widehat{\gamma}\in (\gamma, \gamma^{\gf, p})$, there exist $\widehat{x}\in\dom{\gf}$ and $\ell_2>-\infty$ such that for each $k\in \mathbb{N}$,
$\gf(y^k)+\frac{1}{p\widehat{\gamma}}\Vert \widehat{x}- y^k\Vert^p\geq \ell_2$. Together with \eqref{eq1:level-bound+locally uniform}, this implies
\begin{equation*}
\frac{1}{p\gamma}\Vert x^k- y^k\Vert^p-\frac{1}{p\widehat{\gamma}}\Vert \widehat{x}-y^k\Vert^p\leq \ell_1 -\ell_2.
\end{equation*}
Dividing by $\Vert y^k \Vert^p$ and letting $k \to +\infty$ yields
$
0<\frac{1}{p\gamma}-\frac{1}{p\widehat{\gamma}}\leq 0,
$
a contradiction.
The lower semicontinuity of $\Phi(x, y)$ follows from that of $\gf$.
\\
$\ref{level-bound+locally uniform:proxnonemp}$ 
This follows directly from $\dom{\fgam{\gf}{p}{\gamma}} = \R^n$ and \cite[Theorem~1.17(a)]{Rockafellar09}.\\
$\ref{level-bound+locally uniform:cononx}$ For given $\ov{x}\in \R^n$ and $\ov{y}\in \prox{\gf}{\gamma}{p}(\ov{x})$, the function $\Phi(\ov{y}, x)$ is continuous in $x$. It then follows from \cite[Theorem~1.17~(c)]{Rockafellar09} that $\fgam{\gf}{p}{\gamma}$ is continuous on $\R^n$.
\\
$\ref{level-bound+locally uniform2:conv}$ 
Defining $\Phi(y, x, \gamma) := \gf(y) + \frac{1}{p\gamma}\Vert x-y\Vert^p$, one obtains an argument analogous to Assertion~$\ref{level-bound+locally uniform:cononx}$ showing that $\fgam{\gf}{p}{\gamma}$ depends continuously on $(x,\gamma)$ in $\R^n\times (0, \gamma^{\gf, p})$.
The result then follows from \cite[Theorem~1.17~(b)]{Rockafellar09} and the continuity of $\fgam{\gf}{p}{\gamma}$.
\end{proof}


We now establish favorable relationships between the critical points of $\gf$ and those of $\fgam{\gf}{p}{\gamma}$, which pave the way for algorithmic developments. We begin by recalling some necessary definitions.
  
\begin{definition}[Reference points]\label{def:critic}
Let $p>1$ and $\gh: \R^n \to \Rinf$ be a proper lsc function, and let $\ov{x}\in \dom{\gh}$. Then, $\ov{x}$ is called
\begin{enumerate}[label=(\textbf{\alph*}), font=\normalfont\bfseries, leftmargin=0.7cm]
  \item \label{def:critic:f} a \textit{Fr\'{e}chet critical point} if $0\in \widehat{\partial}\gh(\ov{x})$, denoted by $\ov{x}\in \bs{\rm Fcrit}(\gh)$;
 \item \label{def:critic:m} a \textit{Mordukhovich critical point} if $0\in \partial \gh(\ov{x})$, denoted by $\ov{x}\in \bs{\rm Mcrit}(\gh)$;
   \item \label{def:critic:pfix} a \textit{proximal fixed point} if $\ov{x}\in \prox{\gh}{\gamma}{p}(\ov{x})$, denoted by $\ov{x}\in \bs{\rm Fix}(\prox{\gh}{\gamma}{p})$.
\end{enumerate}
\end{definition}

In the following theorem, we examine the relationships among these notions.
\begin{theorem}[Relationships among reference points]\label{prop:relcrit}
Let $p>1$ and let 
$\gf: \R^n \to \Rinf$ be a proper lsc function. The following statements hold:
\begin{enumerate}[label=(\textbf{\alph*}), font=\normalfont\bfseries, leftmargin=0.7cm]
 \item \label{prop:relcrit:inf:arg:eqinf} For each $\gamma>0$, $\bs\inf_{y\in \R^n}  \fgam{\gf}{p}{\gamma}(y)=\bs\inf_{y\in \R^n} \gf(y)$;
 \item \label{prop:relcrit:inf:arg:argmin} For $\gamma>0$,
 if $\prox{\gf}{\gamma}{p} (x)\neq \emptyset$ for all $x\in \R^n$ and $\bs\inf_{y\in\R^n} \gf(y)\neq -\infty$, then
  \begin{equation*}
  \argmint{y\in \R^n} \fgam{\gf}{p}{\gamma}(y)=\argmint{y\in \R^n} \gf(y);
  \end{equation*}

\item \label{prop:relcrit:FcritMcritfix}
 If $\gf$ is high-order prox-bounded with threshold $\gamma^{\gf, p}>0$  and $\bs\inf_{y\in\R^n} \gf(y)\neq -\infty$, then for each $\gamma\in (0, \gamma^{\gf, p})$, 
\begin{align*}
\argmint{x\in \R^n}\gf(x)\subseteq \bs{\rm Fcrit}(\fgam{\gf}{p}{\gamma})\subseteq \bs{\rm Mcrit}(\fgam{\gf}{p}{\gamma})
&\subseteq \bs{\rm Fix}(\prox{\gf}{\gamma}{p})
\\&\subseteq\bs{\rm Fcrit}(\gf)\subseteq \bs{\rm Mcrit}(\gf).
 \end{align*}
\end{enumerate}
\end{theorem}
\begin{proof}
 $\ref{prop:relcrit:inf:arg:eqinf}$  This follows from \cite[Propositions~12.9]{Bauschke17}.\\
$\ref{prop:relcrit:inf:arg:argmin}$  
From
\cite[Theorem~2.3]{Attouch1989}, we have $\argmin{y\in \R^n} \gf(y)\subseteq \argmin{y\in \R^n} \fgam{\gf}{p}{\gamma}(y)$.
Suppose $\ov{x}\in \argmin{y\in \R^n} \fgam{\gf}{p}{\gamma}(y)$ and let
$\ov{y}\in \prox{\gf}{\gamma}{p} (\ov{x})$. By Assertion~$\ref{prop:relcrit:inf:arg:eqinf}$,
\begin{equation*}
\gf(\ov{y})+\frac{1}{p\gamma}\Vert \ov{x}-\ov{y}\Vert^p=\fgam{\gf}{p}{\gamma}(\ov{x})=\mathop{\bs{\inf}}\limits_{y\in \R^n}  \fgam{\gf}{p}{\gamma}(y)=\mathop{\bs{\inf}}\limits_{y\in \R^n} \gf(y)\leq \gf(\ov{y})\leq \gf(\ov{y})+\frac{1}{p\gamma}\Vert \ov{x}-\ov{y}\Vert^p.
\end{equation*}
Hence, $\ov{x} = \ov{y}$ and $\argmin{y\in \R^n} \fgam{\gf}{p}{\gamma}(y)\subseteq \argmin{y\in \R^n} \gf(y)$, concluding~$\ref{prop:relcrit:inf:arg:argmin}$.
\\
$\ref{prop:relcrit:FcritMcritfix}$
If $\ov{x}\in \argmin{x\in \R^n}\gf(x)$, then invoking Theorem~\ref{th:level-bound+locally uniform} $\ref{level-bound+locally uniform:proxnonemp}$ and Assertion~$\ref{prop:relcrit:inf:arg:argmin}$ implies that $\ov{x}\in \argmin{y\in \R^n} \fgam{\gf}{p}{\gamma}(y)$. 
Hence, by  \cite[Theorem~10.1]{Rockafellar09}, $0\in \widehat{\partial}\fgam{\gf}{p}{\gamma}(\ov{x})$, i.e., 
 $\ov{x}\in \bs{\rm Fcrit}(\fgam{\gf}{p}{\gamma})\subseteq \bs{\rm Mcrit}(\fgam{\gf}{p}{\gamma})$.
 If  $\ov{x}\in \bs{\rm Mcrit}(\fgam{\gf}{p}{\gamma})$ and $\ov{y}\in \prox{\gf}{\gamma}{p}(\bar{x})$, by \cite[Lemma~3.1]{KecisThibault15}, $\ov{x}=\ov{y}$ (see also Remark~\ref{rem:onassum:eps}~$\ref{rem:onassum:eps:b}$). Thus, $\ov{x}\in \bs{\rm Fix}(\prox{\gf}{\gamma}{p})$.
Finally, if $\ov{x}\in  \bs{\rm Fix}(\prox{\gf}{\gamma}{p})$, then by  \cite[Exercise~8.8~\&~Theorem~10.1]{Rockafellar09}, $0\in \widehat{\partial}\gf(\ov{x})$. Hence, $\ov{x}\in \bs{\rm Fcrit}(\gf)\subseteq \bs{\rm Mcrit}(\gf)$.
\end{proof}
 Regarding some reverse relations in Theorem~\ref{prop:relcrit}, we present the following remark.
 \begin{remark}\label{rem:revofrel}
Let us consider the assumptions of Theorem~\ref{prop:relcrit}.
\begin{enumerate}[label=(\textbf{\alph*}), font=\normalfont\bfseries, leftmargin=0.7cm]
\item \label{rem:revofrel:a}
It may happen that $\bs{\rm Fix}(\prox{\gf}{\gamma}{p}) \nsubseteq  \argmin{x \in \R^n} \gf(x)$: For $p=2$, $\gamma=0.5$, and $\gf(x) = \cos(x)$, we have $0 \in \bs{\rm Fix}(\prox{\gf}{\gamma}{p})$, but $0 \notin \argmin{x \in \R^n} \gf(x)$.
  
  \item \label{rem:revofrel:b} It may also happen that $\bs{\rm Fcrit}(\gf) \nsubseteq \bs{\rm Fix}(\prox{\gf}{\gamma}{p})$: See \cite[Example~3.6]{Themelis18}.
 
  \end{enumerate}
  \end{remark}


To facilitate the global and linear convergence of methods for minimizing $\fgam{\gf}{p}{\gamma}$, we next investigate the relation between the KL property of $\fgam{\gf}{p}{\gamma}$ and that of $\gf$.

\begin{theorem}[KL property of HOME]\label{th:KLp:hope}
Let $p>1$ and let $\gf: \R^n \to \Rinf$ be a proper lsc function. If $\gf$  satisfies the KL property with an exponent $\theta\in \left[\frac{p-1}{p}, 1\right)$, then, 
\begin{enumerate}[label=(\textbf{\alph*}), font=\normalfont\bfseries, leftmargin=0.7cm]
\item \label{th:KLp:hope:a} $\Phi(x, y):=\gf(y)+\frac{1}{p\gamma}\Vert x - y\Vert^{p}$ satisfies the  KL property with  exponent $\theta$;
\item \label{th:KLp:hope:b} If  $\gf$ is high-order prox-bounded with threshold $\gamma^{\gf, p}>0$, then for any $\gamma\in (0, \gamma^{\gf, p})$, the function $\fgam{\gf}{p}{\gamma}$ satisfies the KL property with the exponent $\theta$.
\end{enumerate}
\end{theorem}
\begin{proof}
$\ref{th:KLp:hope:a}$
By \cite[Lemma~2.1]{Li18}, $\Phi$ satisfies the KL property for any $\theta\in [0, 1)$ at each $(x, y)\in\Dom{\partial \Phi}$ with $0\notin \partial \Phi(x,y)$;
see Remark~\ref{rem:on KL prop}~$\ref{rem:on KL prop:b}$.
Now consider $(\ov{x}, \ov{y})\in\Dom{\partial \Phi}$ with $0\in \partial \Phi(\ov{x}, \ov{y})$. Since
\begin{equation}\label{eq:KLdefinition:subphi}
\partial \Phi(x,y)=\left[\begin{matrix}
\frac{1}{\gamma}\Vert x -y\Vert^{p-2} (x-y)\\ \partial\gf(y)-\frac{1}{\gamma}\Vert x - y\Vert^{p-2} (x - y)
\end{matrix}\right],
\end{equation}
we must have $\ov{x}=\ov{y}$. 
We show that $\Phi$ satisfies the KL property at $(\ov{x}, \ov{x})\in \Dom{\partial \Phi}$.

By assumption, there exist constants $c, r>0$  and $\eta\in (0, +\infty]$ such that 
\begin{equation}\label{eq:theorem:KLdefinition}
\dist^{\frac{1}{\theta}}(0, \partial \gf(x))\geq c\left\vert\gf(x)-\gf(\ov{x})\right\vert,~~\forall x\in \mb(\ov{x}; r)\cap\Dom{\partial \gf},
\end{equation}
with $0<\left\vert \gf(x)-\gf(\ov{x})\right\vert<\eta$.
Moreover, \eqref{eq:KLdefinition:subphi}, Lemma~\ref{lemma:ineq:inequality p}~$\ref{lemma:ineq:inequality p:ineq1}$, and \cite[Lemma~2.2]{Li18} imply the existence of a constant $c_0>0$, such that for 
$\lambda\in (0,\frac{1}{2})$, 
$\eta_1=\lambda^{p-1}$, and $\eta_2=\left(\frac{\lambda}{1-\lambda}\right)^{p-1}<1$,  for each $(x,y)\in \Dom{\partial \Phi}$, we obtain
\begin{align*}
\dist^{\frac{1}{\theta}}(0, \partial \Phi(x,y))&\geq c_0\left(
(\frac{1}{\gamma})^{\frac{1}{\theta}}\Vert x - y\Vert^{\frac{p-1}{\theta}}+\mathop{\bs{\inf}}\limits_{\zeta\in \partial\gf(y)}\Vert \zeta -\frac{1}{\gamma}\Vert x - y\Vert^{p-2} (x-y)\Vert^{\frac{1}{\theta}}\right)
\\&\geq c_0\left(
(\frac{1}{\gamma})^{\frac{1}{\theta}}\Vert x - y\Vert^{\frac{p-1}{\theta}}+\mathop{\bs{\inf}}\limits_{\zeta\in \partial\gf(y)}\eta_1\Vert \zeta\Vert^{\frac{1}{\theta}}-\eta_2(\frac{1}{\gamma})^{\frac{1}{\theta}}\Vert x - y\Vert^{\frac{p-1}{\theta}}\right)
\\&\geq c_1\left(\mathop{\bs{\inf}}\limits_{\zeta\in \partial\gf(y)}\Vert \zeta\Vert^{\frac{1}{\theta}}+\Vert x - y\Vert^{\frac{p-1}{\theta}}\right),
\end{align*}
where
$c_1:=c_0~\bs\min\{\eta_1, (1-\eta_2)\left(\frac{1}{\gamma}\right)^{\frac{1}{\theta}}\}$.
Choose $r'\in (0, \bs\min\{r, \frac{1}{2}\})$
such that $\frac{1}{p\gamma}(2r')^p<\frac{\eta}{2}$.
Then \eqref{eq:theorem:KLdefinition} holds with $r'$ in place of $r$ and the same constant $c$. 
Let $(x,y)\in \mb((\ov{x}, \ov{x}); r')\cap \Dom{\partial \Phi}$ satisfy $0<\left\vert\Phi(x,y)-\Phi(\ov{x},\ov{x})\right\vert<\frac{\eta}{2}$. Since $\Vert x- y\Vert\leq 2r'$, we have 
$\frac{1}{p\gamma}\Vert x-y\Vert^p<\frac{\eta}{2}$. Moreover,
\begin{align}\label{eq:theorem:KLdefinition:leseta}
\vert \gf(y) - \gf(x)\vert=\left\vert \Phi(x,y) -\Phi(x,x)-\frac{1}{p\gamma}\Vert x-y\Vert^p\right\vert
&\leq \vert \Phi(x,y) -\Phi(x,x)\vert+\frac{1}{p\gamma}\Vert x-y\Vert^p
\notag\\&< \frac{\eta}{2}+\frac{\eta}{2}=\eta.
\end{align}
Moreover,
\begin{align*}
\dist^{\frac{1}{\theta}}\left(0, \partial \Phi(x,y)\right)&
\overset{(i)}{\geq} c_1(\mathop{\bs{\inf}}\limits_{\zeta\in \partial\gf(y)}\Vert \zeta\Vert^{\frac{1}{\theta}}+\Vert x - y\Vert^{p})
\\&\geq c_2(\mathop{\bs{\inf}}\limits_{\zeta\in \partial\gf(y)}c^{-1}\Vert \zeta\Vert^{\frac{1}{\theta}}+\frac{1}{p\gamma}\Vert x - y\Vert^{p}),
\end{align*}
where $c_2:=c_1c~\bs\min \{1,c^{-1}p\gamma\}$ and $(i)$ follows from $\Vert x-y\Vert\leq 2r'<1$ and $\frac{p-1}{\theta}\leq p$.
Together with \eqref{eq:theorem:KLdefinition}  and \eqref{eq:theorem:KLdefinition:leseta}, this yields
\begin{align}\label{eq:theorem:KLdefinitionb}
\dist^{\frac{1}{\theta}}(0, \partial \Phi(x,y))
&\geq c_2(\vert\gf(y)-\gf(\ov{x})\vert+\frac{1}{p\gamma}\Vert x - y\Vert^{p})
\nonumber\\&\geq c_2(\vert\gf(y)-\gf(\ov{x})+\frac{1}{p\gamma}\Vert x - y\Vert^{p}\vert)
=c_2 \vert\Phi(x,y)-\Phi(\ov{x}, \ov{x})\vert,
\end{align}
whenever
$(x,y)\in \mb((\ov{x}, \ov{x}); r')\cap \Dom{\partial \Phi}$ with $0<\left\vert\Phi(x,y)-\Phi(\ov{x},\ov{x})\right\vert<\frac{\eta}{2}$,
i.e., $\Phi$ satisfies the KL property with exponent $\theta$. 
\\
$\ref{th:KLp:hope:b}$ This follows from
Assertion~$\ref{th:KLp:hope:a}$ and \cite[Theorem~3.1]{Yu2022}; see Remark~\ref{rem:on KL prop}~$\ref{rem:on KL prop:c}$.
\end{proof}


\begin{remark}\label{rem:on KL prop} 
\begin{enumerate}[label=(\textbf{\alph*}), font=\normalfont\bfseries, leftmargin=0.4cm]
\item \label{rem:on KL prop:a} 
A similar result to Theorem~\ref{th:KLp:hope} for $p=2$ is discussed in \cite[Remark~5.1~(i)]{Yu2022}, where
the KL property with an exponent is defined without the absolute value
(see Remark~\ref{rem:KLwabsval}).
A similar result to Theorem~\ref{th:KLp:hope}, without considering the absolute value and for $p>1$,
can be obtained by adapting the proof of this theorem.

\item \label{rem:on KL prop:b} 
In the proof of Theorem~\ref{th:KLp:hope}, we used \cite[Lemma~2.1]{Li18}, which defines the KL property in the sense of Remark~\ref{rem:KLwabsval}. Following its argument, one can show that if $\gh: \R^n \to \Rinf$ is proper and lsc, $\ov{x} \in \Dom{\partial \gh}$, and $0 \notin \partial \gh(\ov{x})$, then $\gh$ satisfies the KL property with any exponent $\theta \in [0,1)$ in the sense of  Definition~\ref{def:kldef}.

\item \label{rem:on KL prop:c} 
In \cite[Theorem~3.1]{Yu2022}, the proof is based on the KL property with an exponent,
without the absolute value.
To adapt it to our setting, let $\mathbb{X}=\mathbb{Y}=\R^n$, $F=\Phi(x, y)$ as in Theorem~\ref{th:KLp:hope}, $f=\fgam{\gf}{p}{\gamma}$, and $Y(x)=\prox{\gf}{\gamma}{p}(x)$. 
Due to the high-order prox-boundedness of $\gf$ and Theorem~\ref{th:level-bound+locally uniform}~$\ref{level-bound+locally uniform:levelunif}$, 
the results of \cite[Lemma~2.1]{Yu2022} hold in this context. 
Defining $\Delta:={\ov{x}}\times Y(\ov{x})$ and applying Lemma~\ref{lem:Uniformized KL property} in place of \cite[Lemma~2.2]{Yu2022}, 
we derive an inequality analogous to \cite[Eq.~(3.1)]{Yu2022}, namely 
$\dist\left(0, \partial F(x, y)\right)\geq c\left\vert F(x, y)- f(\ov{x})\right\vert^\theta$, 
whenever $\dist((x,y), \Delta)<\nu$ and $0<\vert F(x,y)-f(\ov{x})\vert<a$, where $a$ and $\nu$ are constants as in the proof of \cite[Theorem~3.1]{Yu2022}. 
Proceeding as in that proof, but replacing the condition $f(\ov{x})<f(x)<f(\ov{x})+\eta$ with $0<\vert f(x)-f(\ov{x})\vert<\eta$, we get $\dist(0, \partial f(x))\geq c\left\vert f(x) - f(\ov{x})\right\vert^\theta$, adjusting our claim.
\end{enumerate}
\end{remark}

\subsection{Inexact oracle for HOME}
\label{subsec:inexacthiord}

We discuss inexactness criteria for computing an approximate element of HOPE (see Assumption~\ref{assum:approx} and Remark~\ref{rem:onassumapprox}). We begin with assumptions that will be used throughout this section.
\begin{assumption}\label{assum:approx} 
Let us consider problem \eqref{eq:mainproblem}. We assume that
\begin{enumerate}[label=(\textbf{\alph*}), font=\normalfont\bfseries, leftmargin=0.7cm]
\item \label{assum:approx:coer weak con} $p>1$ and $\gf: \R^n\to \Rinf$ is a proper, lsc function bounded from below;

\item \label{assum:approx:eps}  $\{\varepsilon_k\}_{k\in \Nz}$ and $\{\delta_k\}_{k\in \Nz}$ are sequences of positive, non-increasing scalars such that $\ov{\varepsilon}:=\sum_{k=0}^{\infty} \varepsilon_k<\infty$ and $\delta_k\downarrow 0$;    
\item \label{assum:approx:aproxep} for a given point $x\in \R^n$, $\gamma>0$, $\varepsilon> 0$, and $\delta> 0$, we can compute
a \textit{prox approximation vector} $\prox{\gf}{\gamma}{p, \varepsilon}(x)\in \R^n$ such that
\begin{align}
&\dist\left(\prox{\gf}{\gamma}{p, \varepsilon}(x), \prox{\gf}{\gamma}{p}(x)\right)\leq \delta,\label{eq:ep-approx:dist}    
\\
&\gf(\prox{\gf}{\gamma}{p, \varepsilon}(x))+\frac{1}{p\gamma}\Vert x-\prox{\gf}{\gamma}{p, \varepsilon}(x)\Vert^p\leq\fgam{\gf}{p}{\gamma}(x)+\varepsilon.\label{eq:ep-approx:fun}
\end{align}
\end{enumerate}
\end{assumption}
Given the prox approximation $\prox{\gf}{\gamma}{p,\varepsilon}(x)$, we define the
\textit{inexact function value} (a zeroth-order oracle for HOME) by
\begin{equation*}
\fgam{\gf}{p,\varepsilon}{\gamma}(x):=\gf(\prox{\gf}{\gamma}{p, \varepsilon}(x))+\frac{1}{p\gamma}\Vert x-\prox{\gf}{\gamma}{p, \varepsilon}(x)\Vert^p.
\end{equation*}
This quantity will serve as the inexact oracle throughout this work.
\begin{remark}\label{rem:onassumapprox}
\begin{enumerate}[label=(\textbf{\alph*}), font=\normalfont\bfseries, leftmargin=0.7cm]
    \item The lower boundedness in Assumption~\ref{assum:approx}~\ref{assum:approx:coer weak con} can be relaxed to high-order prox-boundedness with threshold $\gamma^{\gf,p}$, provided we assume \linebreak$\gamma\in(0,\gamma^{\gf,p})$ wherever relevant.
    
    \item In relations \eqref{eq:ep-approx:dist} and \eqref{eq:ep-approx:fun}, if $\delta=0$ or $\varepsilon=0$, 
     then $\prox{\gf}{\gamma}{p, \varepsilon}(x)\in\prox{\gf}{\gamma}{p}(x)$. 

    \item 
    Regarding \eqref{eq:ep-approx:fun}, by the definition of the infimum, for each $\varepsilon > 0$ there exists $y \in \R^n$ such that
     \begin{equation*}
    \fgam{\gf}{p}{\gamma}(x)\leq\gf(y) + \frac{1}{p\gamma}\|x-y\|^p \leq\fgam{\gf}{p}{\gamma}(x) + \varepsilon.
    \end{equation*}
    Since $\fgam{\gf}{p}{\gamma}(x)<+\infty$ by Theorem~\ref{th:level-bound+locally uniform}~\ref{level-bound+locally uniform:proxnonemp}, any such $y$ must belong to $\dom{\gh}$. Thus, relations~\eqref{eq:ep-approx:dist} and~\eqref{eq:ep-approx:fun} implicitly 
    assume that the prox approximation $\prox{\varphi}{\gamma}{p,\varepsilon}(x)$ lies in $\dom{\gh}$.

        \item \label{rem:onassumapprox:com} From Theorem~\ref{th:level-bound+locally uniform}~$\ref{level-bound+locally uniform:proxnonemp}$, the set  $\prox{\gf}{\gamma}{p}(x)$ is nonempty and compact. Hence, there exists $\ov{y}\in  \prox{\gf}{\gamma}{p}(x)$ such that 
    \begin{equation*}
    \Vert \prox{\gf}{\gamma}{p, \varepsilon}(x)-\ov{y}\Vert= \dist\left(\prox{\gf}{\gamma}{p, \varepsilon}(x), \prox{\gf}{\gamma}{p}(x)\right).
    \end{equation*}
     Although computing $\ov{y}$ may be infeasible in practice, this selection is useful for analyzing the behavior of $\fgam{\gf}{p}{\gamma}$ near cluster points of iterates generated by our algorithms.

    \item \label{rem:onassumapprox:epsdel} 
    In Assumption~\ref{assum:approx}~\ref{assum:approx:aproxep}, we allow two error parameters, $\varepsilon$ and $\delta$, to define a prox approximation. 
In the convex case with $p=2$, \cite[Lemma~3.1]{fukushima1996} shows that if 
 $\prox{\gf}{\gamma}{p, \varepsilon}(x)$ satisfies \eqref{eq:ep-approx:fun}, 
then it automatically satisfies \eqref{eq:ep-approx:dist} with 
$\delta = \sqrt{2\gamma\,\varepsilon}$. 
Analogous implications hold in certain nonconvex settings; see, e.g., 
\cite[Theorem~4]{bonettini2020} and \cite[Proposition~3.15]{Khanh24weak}. 
For such classes, we certify \eqref{eq:ep-approx:fun} in practice by terminating the inner solver 
when a computable duality gap (or valid lower bound) falls below $\varepsilon$; 
see, e.g., \cite{Boroun21,fukushima1996,Khanh24weak,liu2012implementable} and references therein. 
As such, we assume that an approximate proximal point satisfying both 
\eqref{eq:ep-approx:dist} and \eqref{eq:ep-approx:fun} can be produced. Studying the lower-level is out of scope of this study, i.e., we only focus on the outer-level convergence analysis.
 
\end{enumerate}
\end{remark}



\section{Boosted high-order inexact proximal-point algorithm} \label{sec:boosted}
The first natural instance of the ItsOPT framework is HiPPA, an inexact high-order proximal-point method, given by 
\begin{equation}\label{eq:HiPPA}
x^{k+1}= \prox{\gf}{\gamma}{p, \varepsilon_k}(x^k),
\end{equation}
which generalizes the classical inexact proximal-point method; see, e.g., \cite{rockafellar1976monotone}. We subsequently present a boosted variant (Boosted HiPPA) based on a nonmonotone line search. To this end, we first relate the root of the exact residual mapping to that of its inexact counterpart. After establishing well-definedness of Boosted HiPPA, we study subsequential, global, and linear convergence to proximal fixed points of HOME.

Our algorithms aim to find a proximal fixed point of the proximal operator. For nonconvex $\gf$, the set of proximal fixed points lies between the set of minimizers and the set of Fr\'{e}chet and Mordukhovich critical points; see Theorem~\ref{prop:relcrit}~$\ref{prop:relcrit:FcritMcritfix}$. Let us define the set-valued \textit{fixed-point residual mapping} $R_\gamma:\R^n\rightrightarrows\R^n$ given by  $R_\gamma(x):= x - \prox{\gf}{\gamma}{p}(x)$, where finding a proximal fixed point is equivalent to finding $\ov{x}\in\R^n$ with $0\in R_\gamma(\ov{x})$. 
We will construct a sequence $\{x^k\}_{k\in \Nz}$ with $\dist(0, R_\gamma(x^k))\to 0$, ensuring every cluster point is a proximal fixed point. 
Since only $\prox{\gf}{\gamma}{p, \varepsilon_k}(x^k)$ is available, our analysis is based on the \textit{approximated residual function} defined as $R_\gamma^{\varepsilon}(x):= x - \prox{\gf}{\gamma}{p, \varepsilon}(x)$. The following key lemma shows that if
$R_\gamma^{\varepsilon_k}(x^k)\to 0$, then $\dist(0, R_\gamma(x^k))\to 0$.

\begin{lemma}[Fixed-point residual function versus its approximation] \label{lem:relresiduals}
Let Assumption~\ref{assum:approx} hold.
If $\{x^k\}_{k\in \Nz}$ satisfies $R_\gamma^{\varepsilon_k}(x^k)\to 0$,  then
 $\dist(0, R_\gamma(x^k))\to 0$. 
\end{lemma}

\begin{proof}
For each $x^k$, it holds that
\begin{equation*}
 \dist\left(0, R_\gamma(x^k)\right) \leq \Vert x^k-\prox{\gf}{\gamma}{p, \varepsilon_k}(x^k)\Vert+\dist\left(\prox{\gf}{\gamma}{p, \varepsilon_k}(x^k),\prox{\gf}{\gamma}{p}(x^k)\right).
\end{equation*}
As $\delta_k\downarrow 0$, it follows that $R_\gamma^{\varepsilon_k}(x^k)\to 0$  implies  $\dist(0, R_\gamma(x^k))\to 0$. \end{proof}

Although the HiPPA algorithm \eqref{eq:HiPPA} is an interesting algorithm for solving \eqref{eq:mainproblem}, it is well-known that generic proximal-point schemes can be slow in practice; see, e.g., \cite{Guler92,Nesterov2023a,Salzo12}. This motivates a boosted variant that employs a derivative-free nonmonotone line search in a similar vein as the schemes given in  \cite{fukushima1981generalized,fukushima1996,mine1981minimization}, in which the search direction is not necessarily descent. This algorithm is called {\it Boosted HiPPA}, presented in Algorithm~\ref{alg:inexact}.

\begin{algorithm}
\caption{Boosted HiPPA}\label{alg:inexact}
\begin{algorithmic}[1]
\STATE  \textbf{Initialization} Start with $x^0\in \R^n$, $\gamma>0$, $\sigma\in \left(0, \frac{1}{p\gamma}\right)$, $\vartheta\in (0,1)$, and set $k=0$;
\WHILE{stopping criteria are not satisfied}
\STATE  Compute $\ov{x}^{k}=\prox{\gf}{\gamma}{p, \varepsilon_k}(x^k)$ and $R_\gamma^{\varepsilon_k}(x^k)$;  \COMMENT{lower-level}
\STATE \label{alg:hippa:dir} Choose a direction
$d^k\in \R^n$ and set $m=0$;  \COMMENT{upper-level}
\REPEAT\label{alg:hippa:strep}
\STATE  Set $\alpha_k=\vartheta^m$ and $\widehat{x}^{k+1}:=(1-\alpha_k)\ov{x}^{k}+\alpha_k(x^k+d^k),~~m=m+1$;  \COMMENT{upper-level}
\STATE  Compute $\prox{\gf}{\gamma}{p, \varepsilon_{k+1}}(\widehat{x}^{k+1})$ and $\fgamepsko{\gf}{p}{\gamma}(\widehat{x}^{k+1})$;  \COMMENT{lower-level}
 \UNTIL{
\begin{equation}\label{eq:alg:ingrad:upd}
\fgamepsko{\gf}{p}{\gamma}(\widehat{x}^{k+1})\leq
\fgamepsk{\gf}{p}{\gamma}(x^k)-\sigma\Vert R_\gamma^{\varepsilon_k}(x^k)\Vert^{p} +\varepsilon_{k}+\varepsilon_{k+1},
\end{equation}
}\label{alg:hippa:endrep}
\STATE  $x^{k+1}=\widehat{x}^{k+1}$;~$k=k+1$;
\ENDWHILE
\end{algorithmic}
\end{algorithm}

In Step~\ref{alg:hippa:dir}, the search direction $d^k$ is not required to be a descent direction; however, a concrete strategy is proposed in Subsection~\ref{sub:impiss}. Let us emphasize that the term  $\varepsilon_{k} + \varepsilon_{k+1}$ on the right-hand side of \eqref{eq:alg:ingrad:upd} induces nonmonotonicity; see, e.g., \cite{ahookhosh2017efficiency, Amini2014} for background on nonmonotone line searches. Accordingly, we refer to \eqref{eq:alg:ingrad:upd} as an \textit{inexact descent condition}.

The next result establishes the well-definedness of Algorithm~\ref{alg:inexact}, meaning that the inequality \eqref{eq:alg:ingrad:upd} is satisfied at certain points, and the inner loop (Steps~\ref{alg:hippa:strep}-\ref{alg:hippa:endrep}) terminates after a finite number of backtracking steps.
 
\begin{theorem}[Well-definedness of Boosted HiPPA]\label{th:welldefalg}
Let Assumption~\ref{assum:approx} hold.
Fix $\gamma>0$, $\sigma\in \left(0, \frac{1}{p\gamma}\right)$, and $\vartheta\in (0,1)$,
and let $x^k$ be generated by Algorithm~\ref{alg:inexact}. If $\prox{\gf}{\gamma}{p, \varepsilon_k}(x^k)\neq x^k$, then there exists $\ov{m}\in \Nz$
such that the inequality \eqref{eq:alg:ingrad:upd} holds. Consequently, the condition in
Step~\ref{alg:hippa:endrep} is met after finitely many backtracking steps.
\end{theorem}
\begin{proof}
Set $\ov{x}^{k}:=\prox{\gf}{\gamma}{p, \varepsilon_k}(x^k)$. If $\ov{x}^{k}\neq x^k$, then by Fact~\ref{fact:horder:Bauschke17:p12.9}, \eqref{eq:ep-approx:fun}, the fact $\fgam{\gf}{p}{\gamma}(x^{k})\leq \fgamepsk{\gf}{p}{\gamma}(x^k)$, and $\sigma\in \left(0, \frac{1}{p\gamma}\right)$, we obtain
\begin{equation}\label{eq:th:welldefalg}
\fgam{\gf}{p}{\gamma}(\ov{x}^{k})\leq \gf(\ov{x}^{k})\leq\fgam{\gf}{p}{\gamma}(x^{k})-\frac{1}{p\gamma}\Vert R_\gamma^{\varepsilon_k}(x^k)\Vert^p+\varepsilon_k<\fgamepsk{\gf}{p}{\gamma}(x^k)-  \sigma\Vert R_\gamma^{\varepsilon_k}(x^k)\Vert^{p}+\varepsilon_k.
\end{equation}
Since $\vartheta^m\downarrow 0$ as $m\to \infty$, we have $(1-\vartheta^m)\ov{x}^{k}+\vartheta^m(x^k+d^k)\to \ov{x}^{k}$. By continuity of $\fgam{\gf}{p}{\gamma}$ and \eqref{eq:th:welldefalg}, there exists $\ov{m}\in \Nz$ such that
\begin{equation}\label{eq:th:welldefalg:b}
\fgam{\gf}{p}{\gamma}\left((1-\vartheta^{\ov{m}})\ov{x}^{k}+\vartheta^{\ov{m}}(x^k+d^k)\right)\leq \fgamepsk{\gf}{p}{\gamma}(x^k)-  \sigma\Vert R_\gamma^{\varepsilon_k}(x^k)\Vert^{p}+\varepsilon_k.
\end{equation}
Set $\alpha_k=\vartheta^{\ov{m}}$ and $\widehat{x}^{k+1}:=(1-\alpha_k)\ov{x}^{k}+\alpha_k(x^k+d^k)$. By \eqref{eq:ep-approx:fun} and by setting $\widehat{y}:=\prox{\gf}{\gamma}{p, \varepsilon_{k+1}}(\widehat{x}^{k+1})$, we have 
\begin{equation*}\label{eq:th:welldefalg:b2}
\fgamepsko{\gf}{p}{\gamma}(\widehat{x}^{k+1})=\gf(\widehat{y})+\frac{1}{p\gamma}\Vert \widehat{x}^{k+1}-\widehat{y}\Vert^p\leq \fgam{\gf}{p}{\gamma}(\widehat{x}^{k+1})+\varepsilon_{k+1},
\end{equation*}
i.e., $\fgamepsko{\gf}{p}{\gamma}(\widehat{x}^{k+1})-\varepsilon_{k+1}\leq \fgam{\gf}{p}{\gamma}(\widehat{x}^{k+1})$. Together with \eqref{eq:th:welldefalg:b}, it follows that
\[
\fgamepsko{\gf}{p}{\gamma}(\widehat{x}^{k+1})-\varepsilon_{k+1}\leq \fgam{\gf}{p}{\gamma}(\widehat{x}^{k+1})\leq
\fgamepsk{\gf}{p}{\gamma}(x^k)-  \sigma\Vert R_\gamma^{\varepsilon_k}(x^k)\Vert^{p}+\varepsilon_k.
\]
Therefore, \eqref{eq:alg:ingrad:upd} holds, and the inner loop terminates.
\end{proof}

Let $\ov{x}^{k}:=\prox{\gf}{\gamma}{p, \varepsilon_k}(x^k)$ and choose $\ov{y}^k\in \prox{\gf}{\gamma}{p}(x^k)$ such that
\begin{equation}\label{eq:defybar}
  \Vert \ov{x}^{k}-\ov{y}^k\Vert=\dist\left(\ov{x}^{k}, \prox{\gf}{\gamma}{p}(x^k)\right),
\end{equation}
   which is well-defined by Theorem~\ref{th:level-bound+locally uniform}~$\ref{level-bound+locally uniform:proxnonemp}$ (see also Remark~\ref{rem:onassumapprox}~$\ref{rem:onassumapprox:com}$).
    
\begin{theorem}[Subsequential convergence]\label{th:comcon}
Let Assumption~\ref{assum:approx} hold.
Let $\{x^k\}_{k\in \Nz}$ and $\{\ov{x}^{k}\}_{k\in \Nz}$ be generated by Algorithm~\ref{alg:inexact} and $\{\ov{y}^{k}\}_{k\in \Nz}$  satisfy \eqref{eq:defybar}. 
Then,
\begin{enumerate}[label=(\textbf{\alph*}), font=\normalfont\bfseries, leftmargin=0.7cm]
\item \label{th:comcon:resconv} $\sum_{k=0}^{\infty}\Vert R_\gamma^{\varepsilon_k}(x^k)\Vert^{p}<+\infty$ and $R_\gamma^{\varepsilon_k}(x^k)\rightarrow 0$;
\item \label{th:comcon:cluster} the sequences  $\{x^k\}_{k\in \Nz}$, $\{\ov{x}^{k}\}_{k\in \Nz}$, and $\{\ov{y}^{k}\}_{k\in \Nz}$ share the same cluster points, if any, and every cluster point is a proximal fixed point;
 \item \label{th:comcon:funconvb} there exists a finite value $\fv\in \R$ such that
 \begin{equation}\label{eq:th:comcon:limser}
  \lim_{k\to\infty}\fgam{\gf}{p}{\gamma}(x^k)= \lim_{k\to\infty} \gf(\ov{y}^k) = \lim_{k\to\infty}\fgamepsk{\gf}{p}{\gamma}(x^{k})
  =\lim_{k\to\infty}\gf(\ov{x}^k)=\fv.
 \end{equation}
 Moreover, if $\widehat{x}\in \R^n$ is a cluster point of $\{x^k\}_{k\in \Nz}$, then
$\gf(\widehat{x})=\fgam{\gf}{p}{\gamma}(\widehat{x})=\fv$.
 Finally, if $\{x^k\}_{k\in \Nz}$ is bounded, then $\lim_{k\to\infty}\fgam{\gf}{p}{\gamma}(\ov{y}^k)=\fv$;
\item \label{th:comcon:funconstant} 
the functions $\gf$ and $\fgam{\gf}{p}{\gamma}$ are constant on the set of cluster points, if any.
\end{enumerate}
\end{theorem}
\begin{proof}
$\ref{th:comcon:resconv}$ Utilizing the inequality \eqref{eq:alg:ingrad:upd} over the first $K>0$ iterations implies
\begin{align}\label{eq:sum:th:comco}
\sigma \sum_{k=0}^{K-1}&\Vert R_\gamma^{\varepsilon_k}(x^k)\Vert^{p}\leq \fgam{\gf}{p,\varepsilon_0}{\gamma}(x^0)-\fgam{\gf}{p,\varepsilon_K}{\gamma}(x^{K})+ \sum_{k=0}^{K-1}(\varepsilon_k+\varepsilon_{k+1})
\nonumber\\
&\overset{(i)}{\leq} \fgam{\gf}{p}{\gamma}(x^0)+\varepsilon_0-\fgam{\gf}{p}{\gamma}(x^K)+ \sum_{k=0}^{K-1}(\varepsilon_k+\varepsilon_{k+1})
\overset{(ii)}{\leq} \fgam{\gf}{p}{\gamma}(x^0)-\gf^*+ 2\ov{\varepsilon},
\end{align}
where $(i)$ follows from \eqref{eq:ep-approx:fun} and $(ii)$ is obtained from 
Theorem~\ref{prop:relcrit}~$\ref{prop:relcrit:inf:arg:eqinf}$. 
\\
$\ref{th:comcon:cluster}$ From Assertion~$\ref{th:comcon:resconv}$, we obtain $\Vert x^k - \ov{x}^k\Vert=\Vert R_\gamma^{\varepsilon_k}(x^k)\Vert\rightarrow 0$. Moreover, since $\delta_k\downarrow 0$,  \eqref{eq:ep-approx:dist} implies  $\Vert \ov{y}^k - \ov{x}^k\Vert\to 0$. Hence,
  $\{x^k\}_{k\in \Nz}$, $\{\ov{x}^{k}\}_{k\in \Nz}$, and $\{\ov{y}^k\}_{k\in \Nz}$ share the same cluster points, if any. Let $\{x^j\}_{j\in J\subseteq \Nz}$ be a subsequence with limiting point $\widehat{x}$.
For the corresponding subsequence $\{\ov{y}^j\}_{j\in J}$, we also have $\ov{y}^j\to \widehat{x}$.
As such, by Theorem~\ref{th:level-bound+locally uniform}~$\ref{level-bound+locally uniform2:conv}$, it follows that $\widehat{x}\in \prox{\gf}{\gamma}{p}(\widehat{x})$. Thus, $\widehat{x}$ is a proximal fixed point.
 \\
$\ref{th:comcon:funconvb}$ From \eqref{eq:alg:ingrad:upd}, we have
$\fgamepsko{\gf}{p}{\gamma}(x^{k+1})\leq \fgamepsk{\gf}{p}{\gamma}(x^k)+\varepsilon_{k}+\varepsilon_{k+1}$, i.e.,
\begin{equation*}
\fgamepsko{\gf}{p}{\gamma}(x^{k+1})+\sum_{j=k+1}^{\infty}\varepsilon_{j}+\sum_{j=k+2}^{\infty}\varepsilon_{j}\leq \fgamepsk{\gf}{p}{\gamma}(x^k)+\varepsilon_{k}+\varepsilon_{k+1}+\sum_{j=k+1}^{\infty}\varepsilon_{j}+\sum_{j=k+2}^{\infty}\varepsilon_{j}.
\end{equation*}
Defining $\fv_k:=\fgamepsk{\gf}{p}{\gamma}(x^k)+\sum_{j=k}^{\infty}\varepsilon_{j}+\sum_{j=k+1}^{\infty}\varepsilon_{j}$, this implies that $\{\fv_k\}_{k\in \Nz}$ is non-increasing and bounded from below by $\gf^*$, i.e., it  converges to some $\fv\in \R$. Hence,
\begin{equation*}
\lim_{k\to \infty}\fgamepsk{\gf}{p}{\gamma}(x^k)=\lim_{k\to \infty} \left(\fgamepsk{\gf}{p}{\gamma}(x^k)+\sum_{j=k}^{\infty}\varepsilon_{j}+\sum_{j=k+1}^{\infty}\varepsilon_{j}\right)=
\lim_{k\to \infty}\fv_k=\fv.
\end{equation*}
On the other hand, from the proof of Assertion~$\ref{th:comcon:cluster}$, we observed that $\Vert x^k - \ov{x}^k\Vert\to 0$ and $\Vert \ov{y}^k - \ov{x}^k\Vert\to 0$, which implies 
$\Vert x^k - \ov{y}^k\Vert\to 0$. Thus, from \eqref{eq:ep-approx:fun},
\begin{align*}
\gf(\ov{y}^k)+\frac{1}{p\gamma}\Vert x^k - \ov{y}^k\Vert^p&=
\fgam{\gf}{p}{\gamma}(x^k)\leq \fgam{\gf}{p,\varepsilon_k}{\gamma}(x^k)=\gf(\ov{x}^k)+\frac{1}{p\gamma}\Vert R_\gamma^{\varepsilon_k}(x^k)\Vert^p
\\&\leq  \fgam{\gf}{p}{\gamma}(x^k)+\varepsilon_k=\gf(\ov{y}^k)+\frac{1}{p\gamma}\Vert x^k - \ov{y}^k\Vert^p+\varepsilon_k.
\end{align*}
Taking limits as $k \to \infty$, we obtain \eqref{eq:th:comcon:limser}. 
Let $\widehat{x}\in \R^n$ be a cluster point of the subsequence $\{x^j\}_{j\in J\subseteq\Nz}$ (and also $\{\ov{y}^j\}_{j\in J}$). Hence, we have
\begin{equation*}
\gf(\widehat{x})\overset{(i)}{\leq} 
\liminf_{\scriptsize{\begin{matrix}j\to\infty\\j\in J \end{matrix}}}\gf(\ov{y}^j)
\leq \limsup_{\scriptsize{\begin{matrix}j\to\infty\\j\in J \end{matrix}}}\gf(\ov{y}^j)
\overset{(ii)}{\leq} \limsup_{\scriptsize{\begin{matrix}j\to\infty\\j\in J \end{matrix}}} \fgam{\gf}{p}{\gamma}(x^j)\overset{(iii)}{=}\fgam{\gf}{p}{\gamma}(\widehat{x})\leq \gf(\widehat{x}),
\end{equation*}
where $(i)$ obtains as $\gf$ is lsc and $(ii)$ comes from 
$\gf(\ov{y}^j)\leq \gf(\ov{y}^j)+\frac{1}{p\gamma}\Vert x^j- \ov{y}^j\Vert^p=\fgam{\gf}{p}{\gamma}(x^j)$.
For $(iii)$, we use the continuity of $\fgam{\gf}{p}{\gamma}$. Now, from \eqref{eq:th:comcon:limser}, one has
\begin{equation*}
\gf(\widehat{x})=\fgam{\gf}{p}{\gamma}(\widehat{x})=\lim_{\scriptsize{\begin{matrix}j\to\infty\\j\in J \end{matrix}}}\fgam{\gf}{p}{\gamma}(x^j)=\lim_{k\to\infty}\fgam{\gf}{p}{\gamma}(x^k)=\fv.
\end{equation*}
Finally, assume that $\{x^k\}_{k\in \Nz}$ is bounded. 
Hence, for some $r>0$, $\{x^k\}_{k\in \Nz}\subseteq \mb(0;r)$. Let $\tilde{x}\in \dom{\gf}$ be arbitrary and fixed from now on.
Since $\gf$ is bounded from below by some $\ell_0\in \R$, we have 
$\gf(\tilde{x})-\ell_0\geq 0$ and for all $k\in \Nz$, it is concluded that
\begin{equation*}
\ell_0+\frac{1}{p\gamma}\Vert x^k- \ov{y}^{k}\Vert^p\leq   \gf(\ov{y}^{k})+\frac{1}{p\gamma}\Vert x^k- \ov{y}^{k}\Vert^p=\fgam{\gf}{p}{\gamma}(x^k)\leq \gf(\tilde{x})+\frac{1}{p\gamma}\Vert x^k - \tilde{x}\Vert^p.
\end{equation*}
Invoking Lemma~\ref{lemma:ineq:inequality p}~\ref{lemma:ineq:inequality p:ineq4} ensures
\begin{equation*}
\frac{1}{p\gamma}\left(2^{1-p}\Vert \ov{y}^{k}\Vert^p-\Vert x^k\Vert^p\right)\leq\frac{2^{p-1}}{p\gamma}\Vert x^k\Vert^p+\left(\gf(\tilde{x})+\frac{2^{p-1}}{p\gamma}\Vert\tilde{x}\Vert^p-\ell_0\right).
\end{equation*}
Hence, for all $k\in \Nz$,
\begin{equation*}
\Vert \ov{y}^{k}\Vert^p\leq \frac{p\gamma}{2^{1-p}}\left(\frac{2^{p-1}+1}{p\gamma}r^p+\left(\gf(\tilde{x})+\frac{2^{p-1}}{p\gamma}\Vert\tilde{x}\Vert^p-\ell_0\right)\right).
\end{equation*}
Thus, the sequence $\{\ov{y}^{k}\}_{k\in \Nz}$ is also bounded. Let $U$ be a bounded set containing all these sequences. Since $\fgam{\gf}{p}{\gamma}$ is Lipschitz on $U$ with a constant $L$, see \cite[Theorem~3.1~(b)]{KecisThibault15}, we have
\begin{equation*}
\vert \fgam{\gf}{p}{\gamma}(x^k) - \fgam{\gf}{p}{\gamma}(\ov{y}^k)\vert\leq L\Vert x^k- \ov{y}^k\Vert\to 0.
\end{equation*}
Therefore, $\lim_{k\to\infty}\fgam{\gf}{p}{\gamma}(\ov{y}^k)=\lim_{k\to\infty}\fgam{\gf}{p}{\gamma}(x^k)=\fv$. 
\\
$\ref{th:comcon:funconstant}$ 
Let $\widehat{x}\in \R^n$ be a cluster point of the sequence $\{x^k\}_{k\in \Nz}$. 
From Assertion~$\ref{th:comcon:funconvb}$, we have
$\gf(\widehat{x})=\fgam{\gf}{p}{\gamma}(\widehat{x})=\fv$. Since $\widehat{x}$ was chosen arbitrarily, this completes the proof.
\end{proof}

\begin{remark}\label{rem:com}
\begin{enumerate}[label=(\textbf{\alph*}), font=\normalfont\bfseries, leftmargin=0.7cm]
\item \label{rem:comcon:compx} A common stopping criterion in Algorithm~\ref{alg:inexact} involves finding a point $x^k$ such that 
$\Vert R_\gamma^{\varepsilon_k}(x^k)\Vert\leq \epsilon$ for a given tolerance $\epsilon>0$. 
Let the sequence $\{x^k\}_{k\in \Nz}$ be generated by Algorithm~\ref{alg:inexact}. 
Then
from \eqref{eq:sum:th:comco}, we obtain
\begin{equation*}
\sigma K\bs\min_{0\leq k\leq K-1}\Vert R_\gamma^{\varepsilon_k}(x^k)\Vert^{p}\leq \fgam{\gf}{p}{\gamma}(x^0)-\gf^*+ 2\ov{\varepsilon}\leq \gf(x^0)-\gf^*+ 2\ov{\varepsilon}.
\end{equation*}
Thus, it is required that
$\frac{\gf(x^0)-\gf^*+ 2\ov{\varepsilon}}{\sigma K}\leq \epsilon^p$.
Hence, the algorithm terminates within some
$k\leq \frac{\gf(x^0)-\gf^*+2\ov{\varepsilon}}{\sigma\epsilon^p}$.
 In this case, we obtain
\begin{equation*}
\dist(0, R_\gamma(x^k)) \leq \Vert x^k-\prox{\gf}{\gamma}{p, \varepsilon_k}(x^k)\Vert+\dist(\prox{\gf}{\gamma}{p, \varepsilon_k}(x^k),\prox{\gf}{\gamma}{p}(x^k))\leq \epsilon+\delta_k.
\end{equation*}
\item \label{th:comcon:dist} It is possible to derive an upper bound for $\dist(0, \partial \gf(\ov{y}^k))$, 
which serves as an approximation to a Mordukhovich critical point of $\gf$.
Let the point $\ov{x}^{k}$ be returned by Algorithm~\ref{alg:inexact}. From 
$\frac{1}{\gamma}\Vert x^k - \ov{y}^k\Vert^{p-2}(x^k - \ov{y}^k)\in \partial \gf(\ov{y}^k)$ 
and by \eqref{eq:ep-approx:dist}, it obtains
$\Vert x^k - \ov{y}^k\Vert\leq \Vert x^k-\ov{x}^{k}\Vert+\Vert \ov{x}^{k} - \ov{y}^k\Vert\leq \epsilon+\delta_k$.
Thus,
\begin{equation*}
\dist(0, \partial \gf(\ov{y}^k))\leq \frac{1}{\gamma}\Vert x^k - \ov{y}^k\Vert^{p-1}\leq \frac{\left(\epsilon+\delta_k\right)^{p-1}}{\gamma}.
\end{equation*}
\end{enumerate}
\end{remark}

\subsection{Global and linear convergence}
\label{subsec:glpbal}
In this subsection, we investigate sufficient conditions guaranteeing the global and the linear convergence of the sequence $\{x^k\}_{k\in \Nz}$ generated by Algorithm~\ref{alg:inexact}. Let us begin with the next generic result. We use the notation $\ov{x}^{k}:=\prox{\gf}{\gamma}{p, \varepsilon_k}(x^k)$ .

\begin{theorem}[Generic global convergence]\label{th:globconv}
Let Assumption~\ref{assum:approx} hold.
Let $\{x^k\}_{k\in \Nz}$ and  $\{\ov{x}^{k}\}_{k\in \Nz}$ be generated by Algorithm~\ref{alg:inexact} and $\{\ov{y}^{k}\}_{k\in \Nz}$ satisfy \eqref{eq:defybar}. 
If  $\sum_{k=0}^{\infty}\Vert R_\gamma^{\varepsilon_k}(x^k)\Vert<+\infty$ and $\sum_{k=0}^{\infty}\Vert d^k\Vert<+\infty$, then all of these sequences
converge to a proximal fixed point.
\end{theorem}
\begin{proof}
It follows from the definition of $x^{k+1}$ and $R_\gamma^{\varepsilon_k}$ that
\begin{align*}
  \sum_{k=0}^{\infty}\Vert x^{k+1}-x^k\Vert \overset{(i)}{\leq} \sum_{k=0}^{\infty}\Vert R_\gamma^{\varepsilon_k}(x^k)\Vert+\sum_{k=0}^{\infty}\Vert d^k\Vert<+\infty,
\end{align*}
where $(i)$ is obtained by $\alpha_k, (1-\alpha_k)\leq 1$.
Hence, $\{x^k\}_{k\in \Nz}$ is a Cauchy sequence and there exists some $\widehat{x} \in\R^n$ such that $x^k\to \widehat{x}$. Since
 $\Vert R_\gamma^{\varepsilon_k}(x^k)\Vert\rightarrow 0$, $\varepsilon_k\downarrow 0$, and $\delta_k\downarrow 0$, it holds that $\ov{x}^{k}, \ov{y}^{k}\to \widehat{x}$.  From 
 Theorem~\ref{th:level-bound+locally uniform}~$\ref{level-bound+locally uniform2:conv}$, this implies $\widehat{x}\in \prox{\gf}{\gamma}{p}(\widehat{x})$, i.e., $\widehat{x}$ is a proximal fixed point.
\end{proof}

Let us consider the following assumptions regarding the KL property and the error conditions, which are assumed to be true throughout the remainder of this section.

\begin{assumption}\label{assum:eps} 
\begin{enumerate}[label=(\textbf{\alph*}), font=\normalfont\bfseries, leftmargin=0.7cm]
 \item \label{assum:eps:a}  The sequence $\{x^k\}_{k\in \Nz}$ generated by Algorithm~\ref{alg:inexact} is\break bounded and $\Omega(x^k)$ denotes the set of its cluster points.

\item \label{assum:eps:b}  $\fgam{\gf}{p}{\gamma}$ satisfies the KL property at each point of $\Omega(x^k)$
with a desingularizing function $\phi$  with the quasi-additivity property described in 
Lemma~\ref{lem:Uniformized KL property}.
\item \label{assum:eps:c}  $ \sum_{k=0}^{\infty} ([\phi'(w_k)]^{-1})^\frac{1}{p-1}<\infty$ for
$w_k:=\frac{2}{p\gamma}\sum_{j=k}^{\infty}\delta_j^p+3\sum_{j=k}^{\infty}\varepsilon_j$ for $k\in \Nz$.
\end{enumerate}
\end{assumption}
\begin{remark}\label{rem:onassum:eps} 
\begin{enumerate}[label=(\textbf{\alph*}), font=\normalfont\bfseries, leftmargin=0.7cm]

\item \label{rem:onassum:eps:b}  
In Assumption~\ref{assum:eps} and in some results of this section, we assume that $x\in \Dom{\partial \fgam{\gf}{p}{\gamma}}$ for some $x$.
Let us discuss the consequences of this assumption.
In \cite[Lemma~3.1]{KecisThibault15}, it was shown that for each $u\in \R^n$ and $z\in \prox{\gf}{\gamma}{p}(u)$, the inclusion 
$\widehat{\partial}\fgam{\gf}{p}{\gamma}(u)\subseteq \left\{\frac{1}{\gamma} \Vert u - z \Vert^{p-2} (u - z)\right\}$ holds. 
Let $\ov{u}\in\R^n$ and $\ov{\zeta}\in \partial\fgam{\gf}{p}{\gamma}(\ov{u})$. By definition, there exist sequences $u^k\to \ov{u}$ and $\zeta^k\in \widehat{\partial}\fgam{\gf}{p}{\gamma}(u^k)$, with $\fgam{\gf}{p}{\gamma}(u^k)\to \fgam{\gf}{p}{\gamma}(\ov{u})$ and $\zeta^k\to \ov{\zeta}$. For each $k$, we have $\zeta^k=\frac{1}{\gamma} \Vert u^k - z^k \Vert^{p-2} (u^k - z^k)$ with $z^k\in \prox{\gf}{\gamma}{p}(u^k)$.
By Theorem~\ref{th:level-bound+locally uniform}~$\ref{level-bound+locally uniform2:conv}$, the sequence $\{z^k\}_{k\in \mathbb{N}}$ has a cluster point 
$\ov{z}\in \prox{\gf}{\gamma}{p}(\ov{u})$, i.e., $\ov{\zeta} = \frac{1}{\gamma} \Vert \ov{u} - \ov{z} \Vert^{p-2} (\ov{u} - \ov{z})$. Therefore,
$\partial\fgam{\gf}{p}{\gamma}(u)\subseteq \left\{\frac{1}{\gamma} \Vert u - z \Vert^{p-2} (u - z)\right\}$ for each $u\in \R^n$ and $z\in  \prox{\gf}{\gamma}{p}(u)$.
 Now, if $x\in \Dom{\partial \fgam{\gf}{p}{\gamma}}$, then for each $y\in \prox{\gf}{\gamma}{p}(x)$, we have
$\partial\fgam{\gf}{p}{\gamma}(x)=\left\{\frac{1}{\gamma} \Vert x - y \Vert^{p-2} (x - y)\right\}$, which, of course, does not necessarily imply the differentiability of $\fgam{\gf}{p}{\gamma}$ at $x$. Nevertheless, for each $y\in \prox{\gf}{\gamma}{p}(x)$, we get
$\dist\left(0, \partial \fgam{\gf}{p}{\gamma}(x)\right)= \frac{1}{\gamma} \Vert x - y \Vert^{p-1}$.

\item \label{rem:onassum:eps:e} Regarding Assumption~\ref{assum:eps}~$\ref{assum:eps:c}$, we provide an example to demonstrate its feasibility.
Assume $p=3$ and $\phi(t)=t^{1-\frac{p-1}{p}}=t^{\frac{1}{3}}$, which is motivated from 
Theorem~\ref{th:KLp:hope}. Setting
$\varepsilon_k = \frac{1}{2^{k+1}}$ and $\delta_k=\left(\varepsilon_k\right)^\frac{1}{p}$, we obtain 
$w_k=\left(3+\frac{2}{p\gamma}\right)\frac{1}{2^k}$ and $([\phi'(t)]^{-1})^\frac{1}{p-1}=\sqrt{3}t^{\frac{1}{3}}$. Hence,
$\sum_{k=0}^{\infty} ([\phi'(w_k)]^{-1})^\frac{1}{p-1}=c\sum_{k=0}^{\infty} \left(\frac{1}{2^{\frac{1}{3}}}\right)^k<\infty$,
with $c=\sqrt{3}\left(3+\frac{2}{p\gamma}\right)^{\frac{1}{3}}$.
\end{enumerate}
\end{remark}

We now address the global convergence of the algorithm under the KL property.
\begin{theorem}[Global convergence under KL property]\label{th:conKL}
Let Assumptions~\ref{assum:approx} and \ref{assum:eps} hold, the sequence
$\{x^k\}_{k\in \Nz}$ and $\{\ov{x}^{k}\}_{k\in \Nz}$ be generated by Algorithm~\ref{alg:inexact} and $\{\ov{y}^{k}\}_{k\in \Nz}$  satisfy \eqref{eq:defybar}. 
If for some constant $D\geq 0$, we have $\Vert d^k\Vert\leq D \Vert R_\gamma^{\varepsilon_k}(x^k)\Vert$ for all $k$, then these sequences converge to a proximal fixed point.
\end{theorem}
\begin{proof}
By the definition of HOME,
$\fgam{\gf}{p}{\gamma}(x^{k+1})\leq \fgam{\gf}{p, \varepsilon_{k+1}}{\gamma}(x^{k+1})$ and by \eqref{eq:ep-approx:fun}, we have
$\fgamepsk{\gf}{p}{\gamma}(x^k)\leq \fgam{\gf}{p}{\gamma}(x^k)+\varepsilon_k$.
Hence, \eqref{eq:alg:ingrad:upd} implies 
\begin{align}\label{eq:th:conKL:neweq1}
\fgam{\gf}{p}{\gamma}(x^{k+1})\leq\fgamepsko{\gf}{p}{\gamma}(\widehat{x}^{k+1})&\leq
\fgamepsk{\gf}{p}{\gamma}(x^k)-\sigma\Vert R_\gamma^{\varepsilon_k}(x^k)\Vert^{p} +\varepsilon_{k}+\varepsilon_{k+1}
\notag\\
&\leq \fgam{\gf}{p}{\gamma}(x^k)+\varepsilon_k -\sigma\Vert R_\gamma^{\varepsilon_k}(x^k)\Vert^{p} +\varepsilon_{k}+\varepsilon_{k+1}.
\end{align}
On the other hand, for each $k\in \Nz$, \eqref{eq:ep-approx:dist} implies
\begin{equation*}
\Vert x^k - \ov{y}^k\Vert \leq \Vert x^k - \ov{x}^k\Vert +\Vert  \ov{x}^k - \ov{y}^k\Vert\leq\Vert x^k - \ov{x}^k\Vert +\delta_k=\Vert R_\gamma^{\varepsilon_k}(x^k)\Vert +\delta_k.
\end{equation*}
From Lemma~\ref{lemma:ineq:inequality p}~\ref{lemma:ineq:inequality p:ineq4}, we get $2^{1-p}\Vert x^k - \ov{y}^k\Vert^p \leq\Vert R_\gamma^{\varepsilon_k}(x^k)\Vert^p +\delta_k^p$.
Now, \eqref{eq:th:conKL:neweq1} implies
\begin{equation}\label{eq:th:conKL:a}
\fgam{\gf}{p}{\gamma}(x^{k+1})\leq \fgam{\gf}{p}{\gamma}(x^k)+\varepsilon_k -\sigma 2^{1-p} \Vert x^k - \ov{y}^k\Vert^p +\sigma \delta_k^p+\varepsilon_{k}+\varepsilon_{k+1}.
\end{equation}
On the other hand, 
\begin{align}\label{eq:th:conKL:a2}
\Vert x^{k+1}- x^k\Vert&=\Vert (1-\alpha_k)\ov{x}^{k}+\alpha_k(x^k+d^k)-x^k\Vert
\nonumber\\&\overset{(i)}{\leq}  \Vert R_\gamma^{\varepsilon_k}(x^k)\Vert+\Vert d^k\Vert
 \leq (1+D)\Vert R_\gamma^{\varepsilon_k}(x^k)\Vert,
 \end{align}
 where $(i)$ is obtained from $(1-\alpha_k), \alpha_k\leq 1$, leading to
\begin{align*}
\Vert x^{k+1}-x^k\Vert^p \leq (1+D)^p\Vert  x^k - \ov{x}^k\Vert^p&\leq (1+D)^p\left(\Vert  x^k - \ov{y}^k\Vert+\Vert \ov{y}^k - \ov{x}^k\Vert\right)^p
\\&\leq 2^{p-1} (1+D)^p\left(\Vert  x^k - \ov{y}^k\Vert^p+\delta_k^p\right),
\end{align*}
i.e., $
-\Vert  x^k - \ov{y}^k\Vert^p\leq -\frac{2^{1-p}}{(1+D)^p}\Vert x^{k+1}-x^k\Vert^p+\delta_k^p.
$
Using this and the inequalities \eqref{eq:th:conKL:a}, we come to
\begin{align*}
&\fgam{\gf}{p}{\gamma}(x^{k+1})\leq \fgam{\gf}{p}{\gamma}(x^k)+\varepsilon_k -\sigma 2^{-p} \Vert x^k - \ov{y}^k\Vert^p  -\sigma 2^{-p} \Vert x^k - \ov{y}^k\Vert^p+\sigma \delta_k^p+\varepsilon_{k}+\varepsilon_{k+1}
\\&\leq \fgam{\gf}{p}{\gamma}(x^k)-\sigma 2^{-p} \Vert x^k - \ov{y}^k\Vert^p  -\sigma 2^{-p} 
\frac{2^{1-p}}{(1+D)^p}\Vert x^{k+1}-x^k\Vert^p
\\&~~+\sigma 2^{-p} \delta_k^p
+\sigma \delta_k^p+2\varepsilon_{k}+\varepsilon_{k+1}
\\&=\fgam{\gf}{p}{\gamma}(x^k) -\sigma 2^{-p} \Vert x^k - \ov{y}^k\Vert^p  - 
\frac{ 2^{1-2p}\sigma}{(1+D)^p}\Vert x^{k+1}-x^k\Vert^p
+ (1+2^{-p})\sigma \delta_k^p
+2\varepsilon_{k}+\varepsilon_{k+1}.
\end{align*}
Defining $v_k:=(1+2^{-p})\sigma\sum_{j=k}^{\infty}\delta_j^p+2\sum_{j=k}^{\infty}\varepsilon_j+\sum_{j=k+1}^{\infty}\varepsilon_j$, this implies,
\begin{equation}\label{eq:th:conKL:zz1}
\fgam{\gf}{p}{\gamma}(x^{k+1})+v_{k+1}\leq \fgam{\gf}{p}{\gamma}(x^k)+v_k -\sigma 2^{-p} \Vert x^k - \ov{y}^k\Vert^p  - 
\frac{ 2^{1-2p}\sigma}{(1+D)^p}\Vert x^{k+1}-x^k\Vert^p.
\end{equation}
Hence, the Lyapunov $\{\fgam{\gf}{p}{\gamma}(x^k)+v_k\}_{k\in \Nz}$ is non-increasing and bounded from below by $\gf^*$. Thus, it is convergent to some $\fv\in \R$. 
As $\{x^k\}_{k\in \Nz}$ is bounded, the set $\Omega(x^k)$ is nonempty and compact.
Let $\widehat{x}\in \Omega(x^k)$ be arbitrary. 
Hence, from  Theorem~\ref{th:comcon}~$\ref{th:comcon:funconvb}$,
\begin{equation}\label{eq:th:conKL:zz1-1}
\lim_{k\to \infty}\fgam{\gf}{p}{\gamma}(x^k)=\lim_{k\to \infty}\left(\fgam{\gf}{p}{\gamma}(x^k)+v_k\right)=\fgam{\gf}{p}{\gamma}(\widehat{x})=\fv.
\end{equation}
 Additionally, from Theorem~\ref{th:comcon}~$\ref{th:comcon:funconstant}$, $\fgam{\gf}{p}{\gamma}$ is constant on $\Omega(x^k)$.
Hence, from Lemma~\ref{lem:Uniformized KL property},
there exist $r>0$, $\eta>0$, and a desingularizing function $\phi$, which satisfies the quasi-additivity property, such that for all $x\in X$ with
\begin{equation*}
X:=\Dom{\partial \fgam{\gf}{p}{\gamma}}\cap \{x\in \R^n \mid \dist(x, \Omega(x^k))<r\}\cap\{x\in \R^n \mid 0<\vert \fgam{\gf}{p}{\gamma}(x) - \fgam{\gf}{p}{\gamma}(\widehat{x})\vert<\eta\},
\end{equation*}
we obtain
\begin{equation}\label{eq:th:globconv:klI:a}
\phi'\left(\vert \fgam{\gf}{p}{\gamma}(x) - \fgam{\gf}{p}{\gamma}(\widehat{x})\vert\right)\dist\left(0, \partial \fgam{\gf}{p}{\gamma}(x)\right)\geq 1.
\end{equation}
Since $\Omega(x^k)$ is compact and $\gf$ is bounded below,  \cite[Theorem~3.1~(b)]{KecisThibault15} 
implies that $\fgam{\gf}{p}{\gamma}$ is Lipschitz on 
$\widehat{X}:=\{x\in \R^n \mid \dist(x, \Omega(x^k))<r\}$. Hence, $\widehat{X}\subseteq \Dom{\partial \fgam{\gf}{p}{\gamma}}$ (see \cite[Theorem~1.22]{Mordukhovich2018}).
We have $v_k\downarrow 0$ and
Theorem~\ref{th:comcon}~$\ref{th:comcon:resconv}$ and \eqref{eq:th:conKL:a2} yield $\Vert x^{k+1}- x^k\Vert\to 0$, i.e., together with \eqref{eq:th:conKL:zz1-1}, for some $\widehat{k}>0$ and for each $k\geq \widehat{k}$,
$v_k<\frac{\eta}{2}$,
$ \dist(x^k, \Omega(x^k))<r$ and 
$\vert\fgam{\gf}{p}{\gamma}(x^k)-\fgam{\gf}{p}{\gamma}(\widehat{x})\vert<\frac{\eta}{2}$.

Define $\Delta_k:=\phi\left(\fgam{\gf}{p}{\gamma}(x^k)-\fv+v_k\right)$. 
Note that $\fgam{\gf}{p}{\gamma}(x^k)+v_k$ is a non-increasing sequence that tends to $\fv$. Thus, $\fgam{\gf}{p}{\gamma}(x^k)-\fv+v_k\geq 0$ 
and 
\begin{equation*}
\fgam{\gf}{p}{\gamma}(x^k)-\fv+v_k\leq\vert\fgam{\gf}{p}{\gamma}(x^k)-\fv\vert+v_k<\eta.
\end{equation*}
Consequently, $\Delta_k$ is well-defined. Moreover,
\begin{align}\label{eq:th:conKL:z1}
  \Delta_k-\Delta_{k+1} & \overset{(i)}{\geq} \phi'\left(\fgam{\gf}{p}{\gamma}(x^k)-\fv+v_k\right)\left[\fgam{\gf}{p}{\gamma}(x^k)+v_k-\fgam{\gf}{p}{\gamma}(x^{k+1})-v_{k+1}\right]
\nonumber  \\&\overset{(ii)}{\geq}  \phi'\left(\vert\fgam{\gf}{p}{\gamma}(x^k)-\fv\vert+v_k\right)\left[\sigma 2^{-p} \Vert x^k - \ov{y}^k\Vert^p  + 
\frac{ 2^{1-2p}\sigma}{(1+D)^p}\Vert x^{k+1}-x^k\Vert^p\right],
\end{align}
where $(i)$ obtains from the concavity of $\phi$, $(ii)$ is from the monotonically decreasing nature of $\phi'$, and from \eqref{eq:th:conKL:zz1}.
\\
If $\fgam{\gf}{p}{\gamma}(x^k)-\fv=0$, then 
$\phi'\left(\vert\fgam{\gf}{p}{\gamma}(x^k)-\fv\vert+v_k\right)=\phi'\left(v_k\right)$. This and \eqref{eq:th:conKL:z1} ensure
\begin{align}\label{eq:th:conKL:z2}
\sigma 2^{-p} \Vert x^k - \ov{y}^k\Vert^p  &+ \frac{ 2^{1-2p}\sigma}{(1+D)^p}\Vert x^{k+1}-x^k\Vert^p\leq 
\left(\Delta_k-\Delta_{k+1} \right)\left([\phi'(v_k)]^{-1}\right)\nonumber\\
&\overset{(i)}{\leq} \left(\Delta_k-\Delta_{k+1} \right)\left(\frac{1}{\gamma}\Vert x^k-\ov{y}^k\Vert^{p-1}+[\phi'(v_k)]^{-1}\right),
\end{align}
where $(i)$ follows from the property $\phi' > 0$ that together \eqref{eq:th:conKL:z1} implies $\Delta_k - \Delta_{k+1} \geq 0$.
If $\fgam{\gf}{p}{\gamma}(x^k)-\fv\neq 0$, from \eqref{eq:th:conKL:z1} and quasi-additivity property, by increasing $\widehat{k}$ if necessary, there exists some $c_\phi>0$ such that
\begin{align*}
  \Delta_k&-\Delta_{k+1}\\&\geq \frac{\sigma(c_\phi)^{-1}}{[\phi'\left(\vert\fgam{\gf}{p}{\gamma}(x^k)-\fv\vert\right)]^{-1}+[\phi'(v_k)]^{-1}}\left[ 2^{-p} \Vert x^k - \ov{y}^k\Vert^p  + 
\frac{2^{1-2p}}{(1+D)^p}\Vert x^{k+1}-x^k\Vert^p\right].
\end{align*}
Then, from \eqref{eq:th:globconv:klI:a},
\begin{align}\label{eq:th:conKL:z3}
\sigma 2^{-p} \Vert x^k - \ov{y}^k\Vert^p  &+ \frac{ 2^{1-2p}\sigma}{(1+D)^p}\Vert x^{k+1}-x^k\Vert^p
\nonumber\\&\leq c_\phi\left(\Delta_k-\Delta_{k+1} \right)\left([\phi'\left(\vert\fgam{\gf}{p}{\gamma}(x^k)-\fv\vert\right)]^{-1}+[\phi'(v_k)]^{-1}\right)
\nonumber\\&\leq c_\phi\left(\Delta_k-\Delta_{k+1} \right)\left(\dist\left(0, \partial \fgam{\gf}{p}{\gamma}(x^k)\right)+[\phi'(v_k)]^{-1}\right)
\nonumber\\&\overset{(i)}{=}c_\phi\left(\Delta_k-\Delta_{k+1} \right)\left(\frac{1}{\gamma}\Vert x^k-\ov{y}^k\Vert^{p-1}+[\phi'(v_k)]^{-1}\right),
\end{align}
where $(i)$ is explained in Remark~\ref{rem:onassum:eps}~$\ref{rem:onassum:eps:b}$.
Thus, from \eqref{eq:th:conKL:z2} and \eqref{eq:th:conKL:z3}, by setting $\widehat{m}:=\bs\max\{1, c_\phi\}$,  for each $k\geq \widehat{k}$, we obtain
\begin{align}\label{eq:th:conKL:z4}
  \Vert x^k - \ov{y}^k\Vert^p  &+ \frac{ 2^{1-p}}{(1+D)^p}\Vert x^{k+1}-x^k\Vert^p
\nonumber\\&\leq \widehat{m}\sigma^{-1}2^{p}\left(\Delta_k-\Delta_{k+1} \right)\left(\frac{1}{\gamma}\Vert x^k-\ov{y}^k\Vert^{p-1}+[\phi'(v_k)]^{-1}\right),
\end{align}
which by using Lemma~\ref{lemma:ineq:inequality p}~$\ref{lemma:ineq:inequality p:ineq4}$ implies
\begin{align}\label{eq:Refeq}
\Vert x^k - \ov{y}^k\Vert &+ \frac{ 2^{\frac{1-p}{p}}}{1+D}\Vert x^{k+1}-x^k\Vert
\nonumber\\&\leq\left[\frac{2^{2p-1}\widehat{m}}{\sigma\gamma}\left(\Delta_k-\Delta_{k+1} \right)\left(\Vert x^k-\ov{y}^k\Vert^{p-1}+\gamma [\phi'(v_k)]^{-1}\right)\right]^{\frac{1}{p}}.
\end{align}
Invoking Young's inequality, if $a\geq 0$ and $b\geq 0$ and $p, q>1$ such that $\frac{1}{p}+\frac{1}{q}=1$ 
(i.e., $q=\frac{p}{p-1}$) we come to
\begin{equation*}
(ab)^\frac{1}{p}\leq \left(\frac{a^p}{p}+\frac{b^q}{q}\right)^{\frac{1}{p}}\overset{(i)}{\leq} \frac{a}{p^{\frac{1}{p}}}+\frac{b^{\frac{q}{p}}}{q^{\frac{1}{p}}}\leq a+b^{\frac{q}{p}}=a+b^{\frac{1}{p-1}},
\end{equation*}
where $(i)$ comes from \eqref{eq:intrp:p01}.
Now, if $p\geq 2$, then $\frac{1}{p-1}\in (0, 1]$ and \eqref{eq:Refeq} implies
\begin{align*}
\Vert x^k - \ov{y}^k\Vert&+ \frac{ 2^{\frac{1-p}{p}}}{1+D}\Vert x^{k+1}-x^k\Vert\\&\leq\left[\frac{2^{2p-1}\widehat{m}}{\sigma\gamma}\left(\Delta_k-\Delta_{k+1} \right)\left(\Vert x^k-\ov{y}^k\Vert^{p-1}+\gamma [\phi'(v_k)]^{-1}\right)\right]^{\frac{1}{p}}
\\&\leq\frac{2^{2p-1}\widehat{m}}{\sigma\gamma}\left(\Delta_k-\Delta_{k+1} \right)+\left(\Vert x^k-\ov{y}^k\Vert^{p-1}+\gamma [\phi'(v_k)]^{-1}\right)^\frac{1}{p-1}
\\&\leq\frac{2^{2p-1}\widehat{m}}{\sigma\gamma}\left(\Delta_k-\Delta_{k+1} \right)+\Vert x^k-\ov{y}^k\Vert+(\gamma [\phi'(v_k)]^{-1})^\frac{1}{p-1},
\end{align*}
leading to
\begin{align}\label{eq:th:conKL:c}
\Vert x^{k+1}-x^k\Vert\leq \frac{2^{\frac{2p^2-1}{p}}(1+D)\widehat{m}}{\sigma\gamma}\left(\Delta_k-\Delta_{k+1} \right)+ \frac{1+D}{2^{\frac{1-p}{p}}}(\gamma [\phi'(v_k)]^{-1})^\frac{1}{p-1}.
\end{align}
If $p\in (1, 2)$, then $\frac{1}{p-1}>1$ and \eqref{eq:Refeq} implies
\begin{align*}
\Vert x^k - \ov{y}^k\Vert&+ \frac{ 2^{\frac{1-p}{p}}}{1+D}\Vert x^{k+1}-x^k\Vert\\&\leq\left[\frac{2^{2p-1}\widehat{m}}{\sigma\gamma}\left(\Delta_k-\Delta_{k+1} \right)\left(\Vert x^k-\ov{y}^k\Vert^{p-1}+\gamma [\phi'(v_k)]^{-1}\right)\right]^{\frac{1}{p}}
\\&=\left[\frac{2^{p+1}\widehat{m}}{\sigma\gamma}\left(\Delta_k-\Delta_{k+1} \right)2^{p-2}\left(\Vert x^k-\ov{y}^k\Vert^{p-1}+\gamma [\phi'(v_k)]^{-1}\right)\right]^{\frac{1}{p}}
\\&\leq\frac{2^{p+1}\widehat{m}}{\sigma\gamma}\left(\Delta_k-\Delta_{k+1} \right)+2^{\frac{p-2}{p-1}}\left(\Vert x^k-\ov{y}^k\Vert^{p-1}+\gamma [\phi'(v_k)]^{-1}\right)^{\frac{1}{p-1}}
\\&\overset{(i)}{\leq}\frac{2^{p+1}\widehat{m}}{\sigma\gamma}\left(\Delta_k-\Delta_{k+1} \right)+\Vert x^k-\ov{y}^k\Vert+(\gamma [\phi'(v_k)]^{-1})^\frac{1}{p-1},
\end{align*}
where for $(i)$, we use Lemma~\ref{lemma:ineq:inequality p}~$\ref{lemma:ineq:inequality p:ineq4}$. Hence,
\begin{align}\label{eq:th:conKL:d}
\Vert x^{k+1}-x^k\Vert\leq \frac{2^{\frac{p^2+2p-1}{p}}(1+D)\widehat{m}}{\sigma\gamma}\left(\Delta_k-\Delta_{k+1} \right)+ \frac{1+D}{2^{\frac{1-p}{p}}}(\gamma [\phi'(v_k)]^{-1})^\frac{1}{p-1}.
\end{align}
It follows from \eqref{eq:th:conKL:c} and \eqref{eq:th:conKL:d} that for each $k \geq \widehat{k}$ and $p > 1$, 
\begin{equation}\label{eq:th:conKL:e}
\Vert x^{k+1}-x^k\Vert\leq \varpi(p)\left(\Delta_k-\Delta_{k+1} \right)+ \frac{1+D}{2^{\frac{1-p}{p}}}(\gamma [\phi'(v_k)]^{-1})^\frac{1}{p-1},
\end{equation}
where $\varpi(p):=\frac{2^{\frac{p^2+2p-1}{p}}(1+D)\widehat{m}}{\sigma\gamma}$ if $p\in (1, 2)$ and 
$\varpi(p):= \frac{2^{\frac{2p^2-1}{p}}(1+D)\widehat{m}}{\sigma\gamma}$ if $p\geq 2$.
\\
On the other hand, we note that
from the chosen $\sigma$ in Algorithm~\ref{alg:inexact},
$\sigma\leq \frac{1}{p\gamma}$. Now, $p>1$ implies $1+2^{-p}\leq 2$. i.e.,
\[
v_k=(1+2^{-p})\sigma\sum_{j=k}^{\infty}\delta_j^p+2\sum_{j=k}^{\infty}\varepsilon_j+\sum_{j=k+1}^{\infty}\varepsilon_j
\leq \frac{2}{p\gamma}\sum_{j=k}^{\infty}\delta_j^p+3\sum_{j=k}^{\infty}\varepsilon_j=w_k,\]
where $w_k$ was introduced in Assumption~\ref{assum:eps}~$\ref{assum:eps:c}$.~~
From monotonically decreasing of $\phi'$, which comes from concavity of $\phi$, we have $\phi'(w_k)\leq \phi'(v_k)$ that implies $[\phi'(v_k)]^{-1}\leq [\phi'(w_k)]^{-1}$. Together with Assumption~\ref{assum:eps}~$\ref{assum:eps:c}$, this implies that 
$\sum_{k=0}^{\infty}([\phi'(v_k)]^{-1})^\frac{1}{p-1}<\infty$.  From \eqref{eq:th:conKL:e}, we obtain
\begin{align*}
\sum_{k=\widehat{k}}^{\infty}\Vert x^{k+1}-x^k\Vert&\leq \varpi(p)\sum_{k=\widehat{k}}^{\infty}\left(\Delta_k-\Delta_{k+1} \right)+ \frac{1+D}{2^{\frac{1-p}{p}}}\gamma^\frac{1}{p-1} \sum_{k=\widehat{k}}^{\infty}([\phi'(v_k)]^{-1})^\frac{1}{p-1}
\\& \overset{(i)}{=} \varpi(p)\Delta_{\widehat{k}}+ \frac{1+D}{2^{\frac{1-p}{p}}}\gamma^\frac{1}{p-1} \sum_{k=\widehat{k}}^{\infty} ([\phi'(v_k)]^{-1})^\frac{1}{p-1}<\infty,
\end{align*}
where $(i)$  follows from the fact that $\phi$ is continuous and $\fgam{\gf}{p}{\gamma}(x^k) - \fv + v_k\to 0$, which implies $\Delta_k \to 0$.
Hence, $\{x^k\}_{k\in \Nz}$ is a Cauchy sequence and  $x^k\to \widehat{x}$. Since
 $\Vert R_\gamma^{\varepsilon_k}(x^k)\Vert\rightarrow 0$ and $\delta_k\downarrow 0$, we have $\ov{x}^{k}, \ov{y}^{k}\to \widehat{x}$.  From Theorem~\ref{th:level-bound+locally uniform}~$\ref{level-bound+locally uniform2:conv}$, this implies $\widehat{x}\in \prox{\gf}{\gamma}{p}(\widehat{x})$, i.e., $\widehat{x}$ is a proximal fixed point.
\end{proof}

To establish the linear convergence of $\{x^k\}_{k\in\Nz}$, we further consider the next assumption, which will be assumed throughout the remainder of this section.

\begin{assumption}\label{assum:linconv}
Let $\omega\in (0, 1)$ and $\{\beta_k\}_{k\in \Nz}$ be a sequence of non-increasing positive scalars such that $\ov{\beta}:=\sum_{k=0}^{\infty} \beta_k<\infty$. We choose $\delta_k$ and $\varepsilon_k$ such that 
\begin{equation}\label{eq:ep-approx:oper}
\varepsilon_k^{\frac{1}{p}}, \delta_k\leq \bs\min\left\{\beta_k,\omega \Vert x^k -\ov{x}^k\Vert, \mathop{\bs{\min}}\left\{\beta_j \Vert x^i - \ov{x}^i\Vert\mid 0\leq i\leq j\leq k\right\} \right\},
\end{equation}
where $\ov{x}^{i}:=\prox{\gf}{\gamma}{p, \varepsilon_i}(x^i)$.
\end{assumption}

\begin{theorem}[Linear convergence]\label{th:linconv}
Let Assumptions~\ref{assum:approx}, \ref{assum:eps}, and \ref{assum:linconv} hold,
 $\{x^k\}_{k\in \Nz}$ and  $\{\ov{x}^{k}\}_{k\in \Nz}$ be generated by Algorithm~\ref{alg:inexact} and $\{\ov{y}^{k}\}_{k\in \Nz}$ satisfies \eqref{eq:defybar}. 
Assume that for some constant $D\geq 0$, we have $\Vert d^k\Vert\leq D \Vert R_\gamma^{\varepsilon_k}(x^k)\Vert$ for all $k$.
If $\fgam{\gf}{p}{\gamma}$ satisfies the KL inequality with an exponent $\theta=\frac{p-1}{p}$ at each point of $\Omega(x^k)$, then all of these sequences converge R-linearly to a proximal fixed point.
\end{theorem}
\begin{proof}
Theorem~\ref{th:conKL} implies the convergence of $\{x^k\}_{k\in \Nz}$, $\{\ov{x}^{k}\}_{k\in \Nz}$, and $\{\ov{y}^{k}\}_{k\in \Nz}$ to a proximal fixed point $\widehat{x}$.  We define $B_k:=\sum_{i\geq k}\Vert x^i - \ov{x}^i\Vert$. Following a similar reasoning as in \eqref{eq:th:conKL:a2}, we get
\begin{align*}
&\Vert x^k - \widehat{x}\Vert\leq \sum_{i\geq k} \Vert x^{i+1} - x^i\Vert \leq (1+D)B_k,\\
& \Vert \ov{x}^k - \widehat{x}\Vert\leq \sum_{i\geq k} \Vert \ov{x}^{i+1} - \ov{x}^i\Vert
\leq \sum_{i\geq k}
 \left[\Vert \ov{x}^{i+1} - x^{i+1}\Vert+\Vert x^{i+1} - x^{i}\Vert+ \Vert \ov{x}^{i} - x^i \Vert \right]
 \\&\hspace{4.2cm}\leq (3+D)B_k,\\
& \Vert \ov{y}^k - \widehat{x}\Vert\leq \Vert \ov{y}^k - \ov{x}^k\Vert+\Vert \ov{x}^k - \widehat{x}\Vert\leq \delta_k +(3+D)B_k 
\\&\qquad\qquad\leq \omega \Vert x^k - \ov{x}^k\Vert+(3+D)B_k\leq (\omega+3+D)B_k.
\end{align*}
We now show that the sequence $\{B_k\}_{k\in \Nz}$ converges at an asymptotically linear rate. By defining  $\phi(t) = c t^{1 - \theta}$, where $c>0$, and using  \eqref{eq:th:globconv:klI:a}, for enough large $k$, we get
\begin{equation}\label{eq:th:linconv:d1}
\left[\frac{(1-\theta)c}{\gamma}\Vert x^k-\ov{y}^k\Vert^{p-1}\right]^{\frac{1}{\theta}}\geq \vert \fgam{\gf}{p}{\gamma}(x^k) - \fv\vert.
\end{equation}
Let $\Delta_k := \phi\left(\fgam{\gf}{p}{\gamma}(x^k) - \fv + v_k\right)$, as defined in the proof of Theorem~\ref{th:conKL}, i.e.,
\begin{align*}
\Delta_k&= c (\fgam{\gf}{p}{\gamma}(x^k)-\fv+v_k)^{1 - \theta}
\leq c (\vert\fgam{\gf}{p}{\gamma}(x^k)-\fv\vert+v_k)^{1 - \theta}
\\&\overset{(i)}{\leq} c \vert\fgam{\gf}{p}{\gamma}(x^k)-\fv\vert^{1 - \theta}+c v_k^{1 - \theta}
\overset{(ii)}{\leq} c \left[\frac{(1-\theta)c}{\gamma}\Vert x^k-\ov{y}^k\Vert^{p-1}\right]^{\frac{1-\theta}{\theta}}+c v_k^{1 - \theta},
\end{align*}
where $(i)$ follows from \eqref{eq:intrp:p01} and $(ii)$  from \eqref{eq:th:linconv:d1}.
Since $\theta=\frac{p-1}{p}$, we have $\frac{1-\theta}{\theta}=\frac{1}{p-1}$ and $1-\theta= \frac{1}{p}$. Thus,
\begin{align}\label{eq:th:linconv:d2}
\Delta_k\leq  c \left[\frac{(1-\theta)c}{\gamma}\Vert x^k-\ov{y}^k\Vert^{p-1}\right]^{\frac{1}{p-1}}+c v_k^{\frac{1}{p}}
= c \left[\frac{(1-\theta)c}{\gamma}\right]^{\frac{1}{p-1}}\Vert x^k-\ov{y}^k\Vert+c v_k^{\frac{1}{p}}.
\end{align}
On the other hand, since $\frac{1}{p}<1$, from \eqref{eq:intrp:p01}, we get
\begin{align*}
v_k^{\frac{1}{p}}\leq\left[(1+2^{-p})\sigma\sum_{j=k}^{\infty}\delta_j^p\right]^{\frac{1}{p}}+\left[3\sum_{j=k}^{\infty}\varepsilon_j\right]^{\frac{1}{p}}
\leq\left[(1+2^{-p})\sigma\right]^{\frac{1}{p}}\sum_{j=k}^{\infty}\delta_j+3^{\frac{1}{p}}\sum_{j=k}^{\infty}\varepsilon_j^{\frac{1}{p}},
\end{align*}
and from \eqref{eq:ep-approx:oper},
\begin{equation*}
\sum_{j=k}^{\infty}\delta_j\leq \Vert x^k - \ov{x}^k\Vert \sum_{j=k}^{\infty}\beta_j\leq \ov{\beta}\Vert x^k - \ov{x}^k\Vert,\qquad
\sum_{j=k}^{\infty}\varepsilon_j^{\frac{1}{p}}\leq \Vert x^k - \ov{x}^k\Vert \sum_{j=k}^{\infty}\beta_j\leq \ov{\beta}\Vert x^k - \ov{x}^k\Vert.
\end{equation*}
These lead to
\begin{equation}\label{eq:th:linconv:k1}
v_k^{\frac{1}{p}}\leq \left(\left[(1+2^{-p})\sigma\right]^{\frac{1}{p}}+3^{\frac{1}{p}}\right)\ov{\beta}\Vert x^k - \ov{x}^k\Vert.
\end{equation}
Together with \eqref{eq:th:linconv:d2}, this implies
\begin{align}\label{eq:th:linconv:d3}
\Delta_k&\leq  c \left[\frac{(1-\theta)c}{\gamma}\right]^{\frac{1}{p-1}}\Vert x^k-\ov{y}^k\Vert+c\left(\left[(1+2^{-p})\sigma\right]^{\frac{1}{p}}+3^{\frac{1}{p}}\right)\ov{\beta}\Vert x^k - \ov{x}^k\Vert
\nonumber\\&\overset{(i)}{\leq}  c \left(1+\ov{\beta}\right)\left[\frac{(1-\theta)c}{\gamma}\right]^{\frac{1}{p-1}}\Vert x^k-\ov{x}^k\Vert+c\left(\left[(1+2^{-p})\sigma\right]^{\frac{1}{p}}+3^{\frac{1}{p}}\right)\ov{\beta}\Vert x^k - \ov{x}^k\Vert
\nonumber\\&= c_1 \Vert x^k - \ov{x}^k\Vert,
\end{align}
where $(i)$ follows from the inequalities
\begin{equation*}
\Vert x^k - \ov{y}^k\Vert\leq \Vert x^k - \ov{x}^k\Vert+\Vert\ov{x}^k - \ov{y}^k\Vert\leq \Vert x^k - \ov{x}^k\Vert+\delta_k\leq (1+\ov{\beta})\Vert x^k - \ov{x}^k\Vert,
\end{equation*}
and $c_1:= c\left[ \left(1+\ov{\beta}\right)\left[\frac{(1-\theta)c}{\gamma}\right]^{\frac{1}{p-1}}+\left(\left[(1+2^{-p})\sigma\right]^{\frac{1}{p}}+3^{\frac{1}{p}}\right)\ov{\beta}
\right]$.
\\
Moreover, from \eqref{eq:th:conKL:z4}, we have
\begin{align}\label{eq:th:conKL:z5}
  \Vert x^k - \ov{y}^k\Vert^p \leq \sigma^{-1}2^{p}\left(\Delta_k-\Delta_{k+1} \right)\left(\frac{1}{\gamma}\Vert x^k-\ov{y}^k\Vert^{p-1}+[\phi'(v_k)]^{-1}\right).
\end{align}
Note that in  \eqref{eq:th:conKL:z4}, we have the term $\widehat{m}:=\bs\max\{1, c_\phi\}$.  As explained before, when considering the KL property with an exponent, then $c_\phi=1$, which implies $\widehat{m}=1$.
Additionally, it holds that
\begin{equation*}
\Vert x^k - \ov{x}^k\Vert\leq \Vert x^k - \ov{y}^k\Vert+\Vert \ov{y}^k - \ov{x}^k\Vert\leq \Vert x^k - \ov{y}^k\Vert+\delta_k\leq \Vert x^k - \ov{y}^k\Vert+\omega
\Vert x^k - \ov{x}^k\Vert. 
\end{equation*}
Thus, $(1-\omega)\Vert x^k - \ov{x}^k\Vert\leq \Vert x^k - \ov{y}^k\Vert$. Together with \eqref{eq:th:linconv:k1}, this leads to
\begin{align}\label{eq:th:conKL:z6}
[\phi'(v_k)]^{-1}&=\frac{1}{c(1-\theta)}v_k^{\theta}=\frac{1}{c(1-\theta)}v_k^{\frac{p-1}{p}}
\nonumber\\&\leq \frac{\left(\left[(1+2^{-p})\sigma\right]^{\frac{1}{p}}+3^{\frac{1}{p}}\right)^{p-1}\ov{\beta}^{p-1}}{c(1-\theta)} \Vert x^k - \ov{x}^k\Vert^{p-1}
\leq c_2 \Vert x^k - \ov{y}^k\Vert^{p-1},
\end{align}
with $c_2:=\frac{\left(\left[(1+2^{-p})\sigma\right]^{\frac{1}{p}}+3^{\frac{1}{p}}\right)^{p-1}\ov{\beta}^{p-1}}{c(1-\theta)(1-\omega)^{p-1}}$.
From \eqref{eq:th:conKL:z5} and \eqref{eq:th:conKL:z6}, we obtain
\begin{align*}
  \Vert x^k - \ov{y}^k\Vert^p \leq \sigma^{-1}2^{p}\left(\Delta_k-\Delta_{k+1} \right)\left(\left(\frac{1}{\gamma}+c_2\right)\Vert x^k-\ov{y}^k\Vert^{p-1}\right),
\end{align*}
and thus,
\begin{align}\label{eq:th:conKL:z7}
 c_3 \Vert x^k - \ov{y}^k\Vert \leq  \Delta_k-\Delta_{k+1},
\end{align}
with $c_3=\sigma2^{-p}\left(\frac{1}{\gamma}+c_2\right)^{-1}$. 
Together with \eqref{eq:th:linconv:d3} and $(1-\omega)\Vert x^k - \ov{x}^k\Vert\leq \Vert x^k - \ov{y}^k\Vert$, this ensures
\begin{align*}
  B_k=\sum_{i\geq k}\Vert x^i - \ov{x}^i\Vert &\leq \frac{1}{1-\omega}\sum_{i\geq k} \Vert x^i - \ov{y}^i\Vert
\leq \frac{1}{c_3(1-\omega)}\sum_{i\geq k} \left(\Delta_i-\Delta_{i+1} \right)
\\
&\overset{(i)}{\leq} \frac{1}{c_3(1-\omega)}\Delta_k\leq c_4 \Vert x^k - \ov{x}^k\Vert\leq  c_4 \left(B_k-B_{k+1}\right),
\end{align*}
where $c_4=\frac{c_1}{c_3(1-\omega)}$ and $(i)$  follows from the fact that $\phi$ is continuous and $\fgam{\gf}{p}{\gamma}(x^k) - \fv + v_k\to 0$, which implies $\Delta_k \to 0$.
Thus, we have $B_{k+1}\leq \left(1-\frac{1}{c_4}\right)B_k$, demonstrating the desired asymptotic $Q$-linear convergence of the sequence $\{B_k\}_{k\in \Nz}$.
Hence, $\{x^k\}_{k\in \Nz}$, $\{\ov{x}^{k}\}_{k\in \Nz}$, and $\{\ov{y}^{k}\}_{k\in \Nz}$ converge R-linearly to a proximal fixed point.
\end{proof}

\begin{remark}
The linear convergence rate established in Theorem~\ref{th:linconv} concerns the \textit{outer-level} of ItsOPT (e.g., Boosted HiPPA). Since the proximal subproblems are solved inexactly, the total computational cost also depends on the number of inner iterations required to satisfy conditions \eqref{eq:ep-approx:dist}–\eqref{eq:ep-approx:fun}. In practice, these subproblems can be handled by inexpensive routines such as subgradient \cite{Davis2018, Rahimi2024}, Bregman gradient \cite{Ahookhosh24}, or BELLA \cite{Ahookhosh21} methods, which typically require only a modest number of iterations. A detailed complexity analysis of the inner iterations is beyond the scope of this paper and will be pursued in future work.
\end{remark}


\begin{corollary}[Global and linear convergence under the KL property of $\gf$]\label{cor:linconv}
Assume that the original function $\gf$ satisfies the KL property with an exponent $\theta\in (0, 1)$ at each point of $\Omega(x^k)$. By setting $p=\frac{1}{1-\theta}$ and under Assumption~\ref{assum:approx}, let $\{x^k\}_{k\in \Nz}$ and $\{\ov{x}^{k}\}_{k\in \Nz}$ be generated by Algorithm~\ref{alg:inexact} and $\{\ov{y}^{k}\}_{k\in \Nz}$ satisfy \eqref{eq:defybar}. 
If $\{x^k\}_{k\in \Nz}$ is bounded, for some constant $D\geq 0$, we have $\Vert d^k\Vert\leq D \Vert R_\gamma^{\varepsilon_k}(x^k)\Vert$ for all $k$,
and the sequences  $\{\varepsilon_k\}_{k\in \Nz}$ and $\{\delta_k\}_{k\in \Nz}$ satisfy Assumption~\ref{assum:linconv}, along with the conditions
\begin{equation*}
\sum_{k=0}^{\infty}\sum_{j\geq k}\delta_j<\infty, \quad \sum_{k=0}^{\infty}\sum_{j\geq k}\varepsilon_j^{\frac{1}{p}}<\infty,
\end{equation*}
then, all of these sequences converge globally and R-linearly to a proximal fixed point.
\end{corollary}
\begin{proof}
Let $\gf$ satisfy the KL property with the desingularizing function $\phi(t)=ct^{1-\theta}$ for some $c>0$
at each point of $\Omega(x^k)$.
Since $\theta\in (0, 1)$  and $p=\frac{1}{1-\theta}$, we obtain $\theta=\frac{p-1}{p}$ and $p>1$. By Theorem~\ref{th:KLp:hope}~$\ref{th:KLp:hope:b}$, the function $\fgam{\gf}{p}{\gamma}$ satisfies the KL property with the exponent $\theta$ on $\Omega(x^k)$.
Defining $w_k:=\frac{2}{p\gamma}\sum_{j=k}^{\infty}\delta_j^p+3\sum_{j=k}^{\infty}\varepsilon_j$ for $k\in \Nz$, we get
\begin{align*}
&\sum_{k=0}^{\infty} \left([\phi'(w_k)]^{-1}\right)^{\frac{1}{p-1}}
= \sum_{k=0}^{\infty} \left(\frac{w_k^\theta}{c(1-\theta)}\right)^{\frac{1}{p-1}}
=\sum_{k=0}^{\infty} \left(\frac{p}{c}w_k^{\frac{p-1}{p}}\right)^{\frac{1}{p-1}}
\\&=\sum_{k=0}^{\infty} \left(\frac{p}{c}\right)^{\frac{1}{p-1}}\left(\frac{2}{p\gamma}\sum_{j=k}^{\infty}\delta_j^p+3\sum_{j=k}^{\infty}\varepsilon_j\right)^{\frac{1}{p}}
\\&\leq \sum_{k=0}^{\infty} \left(\frac{p}{c}\right)^{\frac{1}{p-1}}\left((\frac{2}{p\gamma})^{\frac{1}{p}}
\sum_{j=k}^{\infty}\delta_j+3^{\frac{1}{p}}\sum_{j=k}^{\infty}\varepsilon_j^{\frac{1}{p}}\right)<\infty.
\end{align*}
Consequently, all parts of Assumption~\ref{assum:eps} are satisfied, and the claim follows from Theorems~\ref{th:conKL}~and~\ref{th:linconv}.
\end{proof}

It is notable that the state-of-the-art optimization methods achieve the linear convergence rate if $\theta\in \left(0,\tfrac{1}{2}\right]$, see, e.g., \cite{Ahookhosh21,attouch2010proximal,Li18,Themelis18,Yu2022}; however, the above results show that if $\varphi$ is a KL function with exponent $\theta\in (0,1)$, one can by leveraging a suitable regularization with $p=\frac{1}{1-\theta}$ guarantee the linear convergence of Boosted HiPPA. To the best of our knowledge, this is the first algorithm guaranteeing linear convergence for any KL function with exponent $\theta\in (0,1)$.

\section{Preliminary numerical experiments}
\label{sec:numerical}
Here, we report some preliminary numerical results to investigate the numerical behavior of Boosted HiPPA on some robust low-rank matrix recovery problems. To this end, we first establish some implementation issues that are necessary for an efficient implementation of this algorithm.

\subsection{Implementation issues}\label{sub:impiss}
This section addresses implementation issues related to the Boosted HiPPA described in Algorithm~\ref{alg:inexact}, which are presented next.

\begin{description}[wide, labelwidth=!, labelindent=0pt]
\item [1.] \textbf{(Approximation of HOPE)}  
We apply the subgradient method with geometrically decaying step-sizes (SG-DSS) for weakly convex functions \cite{Davis2018,Li20Low,Rahimi2024} to find a prox approximation $\prox{\gf}{\gamma}{p, \varepsilon_k}(x^k)$; see also \cite[Proposition~3.14]{Khanh24weak}.
Specifically, in the $k$th iteration of Boosted HiPPA, we perform $I_k$  iterations of  SG-DSS to solve the subproblem
 $\min_{y\in \R^n} \left(\gf(y)+\frac{1}{p\gamma}\Vert x^{k-1}- y\Vert^p\right)$ and obtain $\ov{x}^k$ as the solution.
For $k=1, 2$, we set $I_k=50$; for $3\leq k<20$, $I_k=300$; for $20\leq k<30$, $I_k=500$, and for each $k\geq 30$, $I_k=800$.
To update the step-size in SG-DSS, we employ the scheme $\alpha_i = \lambda q^i$, where $\lambda = 1$ and $q = 0.93$.
 Furthermore, we select the sequence $\{\varepsilon_k=\frac{1}{(k+1)^2}\}_{k\in \Nz}$ for the update step in \eqref{eq:alg:ingrad:upd}.
 
  \item [2.] \textbf{(Search direction)} 
 We define the search direction $d^k:=-\sigma_k R_\gamma^{\varepsilon_k}(x^k)$ at each iteration \cite{laCruz2006spectral}. To determine $\sigma_k$, we first introduce the constants $\sigma_{\text{min}} = 10^{-1}$ and $\sigma_{\text{max}} = 10^{10}$, and initialize with $\sigma_0 = 1$. Then, we compute
$\widehat{\sigma}_k=\frac{\langle s^k, s^k\rangle}{\langle s^k, y^k\rangle}$, where
$s^k=x^k - x^{k-1}$ and $y^k=R_\gamma^{\varepsilon_k}(x^k) - R_\gamma^{\varepsilon_{k-1}}(x^{k-1})$. If $\vert \widehat{\sigma}_k\vert\in [\sigma_{min}, \sigma_{max}]$, we set $\sigma_k=\vert \widehat{\sigma}_k\vert$. Otherwise, 
\begin{equation*}
  \sigma_k=\begin{cases}
                    1 &   \mbox{if } \Vert R_\gamma^{\varepsilon_k}(x^k)\Vert>1, \\
                    10^5 &  \mbox{if }  \Vert R_\gamma^{\varepsilon_k}(x^k)\Vert<10^{-5}, \\
                    \Vert R_\gamma^{\varepsilon_k}(x^k)\Vert^{-1} & \mbox{otherwise}.
                  \end{cases}
\end{equation*}

\item [3.]  \textbf{(Parameters of Boosted HiPPA)} We compare the values of $p\in \{1.25, 2, 3\}$, running the algorithm for $15$ and $60$ seconds for problems \eqref{eq:Experiments:optmodel} and \eqref{eq:Experiments:optmodel2}, respectively, in the following. The values of $\gamma$ and $\vartheta$   in Algorithm~\ref{alg:inexact} are determined through tuning and will be reported below. We set $\sigma =\frac{1}{1.1p\gamma}$.

\item [4.] \textbf{(Implementation and comparisons)} We compare Boosted HiPPA with SG-DSS, 
the subgradient method with constant stepsize (\textit{SG-CSS}) for stepsizes $\alpha\in\{0.01, 0.1, 1\}$, and Polyak subgradient method (\textit{SG-PSS}) for weakly convex functions \cite{Davis2018,Rahimi2024}. Furthermore, we compare our method with the high-order inexact proximal algorithm (\textit{HiPPA}) given as
 $x^{k+1}=\prox{\gf}{\gamma}{p, \varepsilon_k}(x^k)$. The algorithm was executed on a laptop with a 12th Gen Intel$\circledR$~Core$^{\text{TM}}$ i7-12800H CPU (1.80 GHz) and 16 GB of RAM, using MATLAB R2022a.
\end{description}

\begin{remark}
We emphasize that, although our theoretical results do not require explicit knowledge of how the high-order proximity problem \eqref{eq:Hiorder-Moreau prox} is solved, the practical implementation of our approach relies critically on the ability to compute approximate solutions to this subproblem efficiently. A natural class of methods for addressing this subproblem is subgradient-based algorithms; however, their convergence guarantees are generally limited in nonconvex settings. Notably, subgradient methods are known to converge for several important classes of nonconvex functions, including weakly convex functions \cite{Davis2018,Li20Low}, paraconvex functions \cite{Rahimi2024}, relative weakly convex functions \cite{rahimi2025subgradient}, upper-$\mathcal{C}^2$ functions \cite{aragon2025nonmonotone}, nonconvex path-differentiable functions \cite{bolte2023subgradient,bolte2022long}, weakly subdifferentiable functions \cite{dinc2021weak}, difference-of-convex functions \cite{Khamaru2018Convergence}, quasiconvex functions \cite{Kiwiel2001Convergence}, and tame functions \cite{bolte2025inexact,Davis2020Stochastic}. Although our numerical experiments in the next section focus on the weakly convex case, the proposed methods are applicable to a substantially broader class of nonconvex functions. As an example, let us consider the function $\gh:\R\to \R$ given by $\gh(x):=\bigl(1 - \max\{x, 0\}\bigr)^2$, which is nonsmooth and non-weakly-convex. Nevertheless, for any fixed $x$, the objective of the proximal subproblem
\[\Psi(y):=\bigl(1 - \max\{y, 0\}\bigr)^2+\frac{1}{p\gamma}\Vert x-y\Vert^p,\]
is a semialgebraic, i.e., the subgradient methods \cite{bolte2025inexact,Davis2020Stochastic} can be applied to this subproblem, generating convergent subsequences under suitable conditions.

\end{remark}

\subsection{Application to robust low-rank matrix recovery} \label{subsec:numerical:matrix}
Matrix recovery is a fundamental problem in signal processing, where the objective is to reconstruct matrices from incomplete or corrupted data. Robust low-rank matrix recovery, in particular, focuses on recovering a low-rank matrix from observations contaminated by sparse errors or noise. 
In this subsection, we demonstrate the application of Boosted HiPPA to address the problem of \textit{robust low-rank matrix recovery}, where outliers, i.e., corrupt the measurements, 
$
  y=\mathcal{A}(X)+s,
$
where $\mathcal{A}: \R^{n_1\times n_2}\to \R^m$ is a known linear operator, $X\in \R^{n_1\times n_2}$ is the matrix to be recovered, and $s\in \R^m$ represents outliers or noise. Such problems frequently arise in applications such as sensor calibration, face recognition, and video surveillance.

Due to the sensitivity of the loss function $\ell_2$ to outliers, an optimization problem of the form $\bs\min_{U\in \R^{n\times r}}\frac{1}{m}\Vert y - \mathcal{A}(UU^T)\Vert_2^2,$
where $X=UU^T$, may produce solutions that are perturbed away from the underlying low-rank matrix in the presence of outliers. In contrast, the $\ell_1$-loss function is more robust against such corruption. 
To address this issue, we consider two models from \cite{Li20Low}. Assuming $n_1 = n_2 = n$ and using the matrix factorization $X = UU^T$ with $U\in \R^{n\times r}$, the first model is formulated as:
\begin{equation}\label{eq:Experiments:optmodel}
    \mathop{\min}\limits_{U\in \mathbb{R}^{n\times r}} \Phi(U) = \frac{1}{m} \Vert y - \mathcal{A}(UU^T) \Vert_1.
\end{equation}
The second model, which uses the factorization $X = UV^T$ with $U \in \R^{n_1 \times r}$ and $V \in \R^{n_2 \times r}$, is formulated as:
\begin{equation}\label{eq:Experiments:optmodel2}
\mathop{\bs{\min}}\limits_{\substack{U \in \mathbb{R}^{n_1 \times r} \\ V \in \mathbb{R}^{n_2 \times r}}} 
\left\{ 
\Psi(U,V) := \frac{1}{m} \Vert y - \mathcal{A}(UV^T) \Vert_1 + \lambda \Vert U^T U - V^T V \Vert_F
\right\},
\end{equation}
where $\Vert \cdot \Vert_F$ denotes the Frobenius norm and $\lambda$ is a regularization parameter.

The parameters for \eqref{eq:Experiments:optmodel} and \eqref{eq:Experiments:optmodel2}, except for $n_1$ and $n_2$ in \eqref{eq:Experiments:optmodel2}, are chosen based on \cite{Li20Low}. Both problems are weakly convex, a subclass of prox-regular functions, as demonstrated in \cite[Propositions~3~and~6]{Li20Low}.
In both models,  $m = 5 r\bs\max\{n_1, n_2\}$ matrices $A_1, \ldots, A_m \in \R^{n_1 \times n_2}$ with i.i.d. standard Gaussian entries define the operator $\mathcal{A}$, and the measurements are computed as $y_i=\langle A_i, X_t\rangle+s_i$, where  $\langle A_i, X_t\rangle=\bs{\rm tr}(X_t^TA_i)$ and $\bs{\rm tr}(\cdot)$  denotes the trace operator. The outlier vector $s$ is generated by randomly selecting $o \times m$ entries, where $o = 0.3$ is the outlier ratio, and populating them with i.i.d. Gaussian values with mean $0$ and variance $100$. The remaining entries are set to $0$.

For model \eqref{eq:Experiments:optmodel}, we set $n=50$ and $r=5$, generating the true matrix $U_t\in \R^{n\times r}$  with i.i.d. standard Gaussian entries and defining $X_t = U_t U_t^T$. We choose $\gamma = 0.5$ and $\vartheta = 0.8$ in Algorithm~\ref{alg:inexact}. As shown in Figure~\ref{fig:ex:diffp}~(a), selecting $p = 1.25$ yields superior results. Consequently, we compare the performance of Boosted HiPPA with HiPPA and subgradient methods using these parameter settings, as illustrated in Figure~\ref{fig:ex:diffp}~(b). The results underscore the improved performance of Boosted HiPPA, alongside the effectiveness of HiPPA.
For model \eqref{eq:Experiments:optmodel2}, we set $n_1 = 50$, $n_2 = 40$, and $r = 5$. The true matrices $U_t \in \R^{n_1 \times r}$ and $V_t \in \R^{n_2 \times r}$ are generated with i.i.d. standard Gaussian entries, and we define $X_t = U_t V_t^T$. To determine $\lambda$, we first select $\delta = \frac{1}{6.3} \sqrt{\frac{2}{\pi}}$, which, together with the outlier ratio $o$, satisfies the conditions in \cite[Proposition~2]{Li20Low}. We then set
$
\lambda = \frac{2(1 - o)\left(\sqrt{\frac{2}{\pi}} - \delta\right) - \left(\sqrt{\frac{2}{\pi}} + \delta\right)}{2}.
$
We use $\gamma = 1$ and $\vartheta = 0.8$ in Algorithm~\ref{alg:inexact}. As illustrated in Figures~\ref{fig:ex:diffp2}~(a)~and~(b), Boosted HiPPA demonstrates superior performance compared to the other methods. 

In Subfigures~(a) and (c) of Figures~\ref{fig:ex:diffp} and \ref{fig:ex:diffp2}, we illustrate the norm of the approximated residual $\Vert R_\gamma^{\varepsilon_k}(x^k) \Vert$ for different $p$ values in HiPPA and Boosted HiPPA, applied to problems \eqref{eq:Experiments:optmodel} and \eqref{eq:Experiments:optmodel2}, respectively. As shown, Boosted HiPPA achieves smaller residual norms in fewer iterations compared to HiPPA, across all tested values of $p$. However, with respect to time, HiPPA for problem~\eqref{eq:Experiments:optmodel}  is faster than the boosted version as it needs fewer computations.
Moreover, the subfigures~(b) of Figures~\ref{fig:ex:diffp} and~\ref{fig:ex:diffp2} show a faster decrease in the objective values $\Phi(\cdot)$ and $\Psi(\cdot)$ for problems~\eqref{eq:Experiments:optmodel} and~\eqref{eq:Experiments:optmodel2}, respectively, when using Boosted HiPPA and HiPPA compared to the other considered subgradient methods, in terms of iteration count. Although subgradient methods are employed to solve the proximal subproblems, Subfigure~(d) of Figures~\ref{fig:ex:diffp} and~\ref{fig:ex:diffp2} illustrates that the CPU times of Boosted HiPPA and HiPPA are comparable to those of the subgradient methods. In particular, for problem~\eqref{eq:Experiments:optmodel}, HiPPA approaches the minimum value significantly earlier than Boosted HiPPA and SG-CCS ($\alpha=0.01$). These observations underscore the need for developing more effective strategies to address the proximal subproblems in inexact high-order proximal point methods, which we reserve for future research.

\begin{figure}[H]
    \begin{center}
        \subfloat[$\Vert R_\gamma^{\varepsilon_k}(x^k)\Vert$ vs. iterations for \eqref{eq:Experiments:optmodel}]
            {\includegraphics[width=5.9cm]{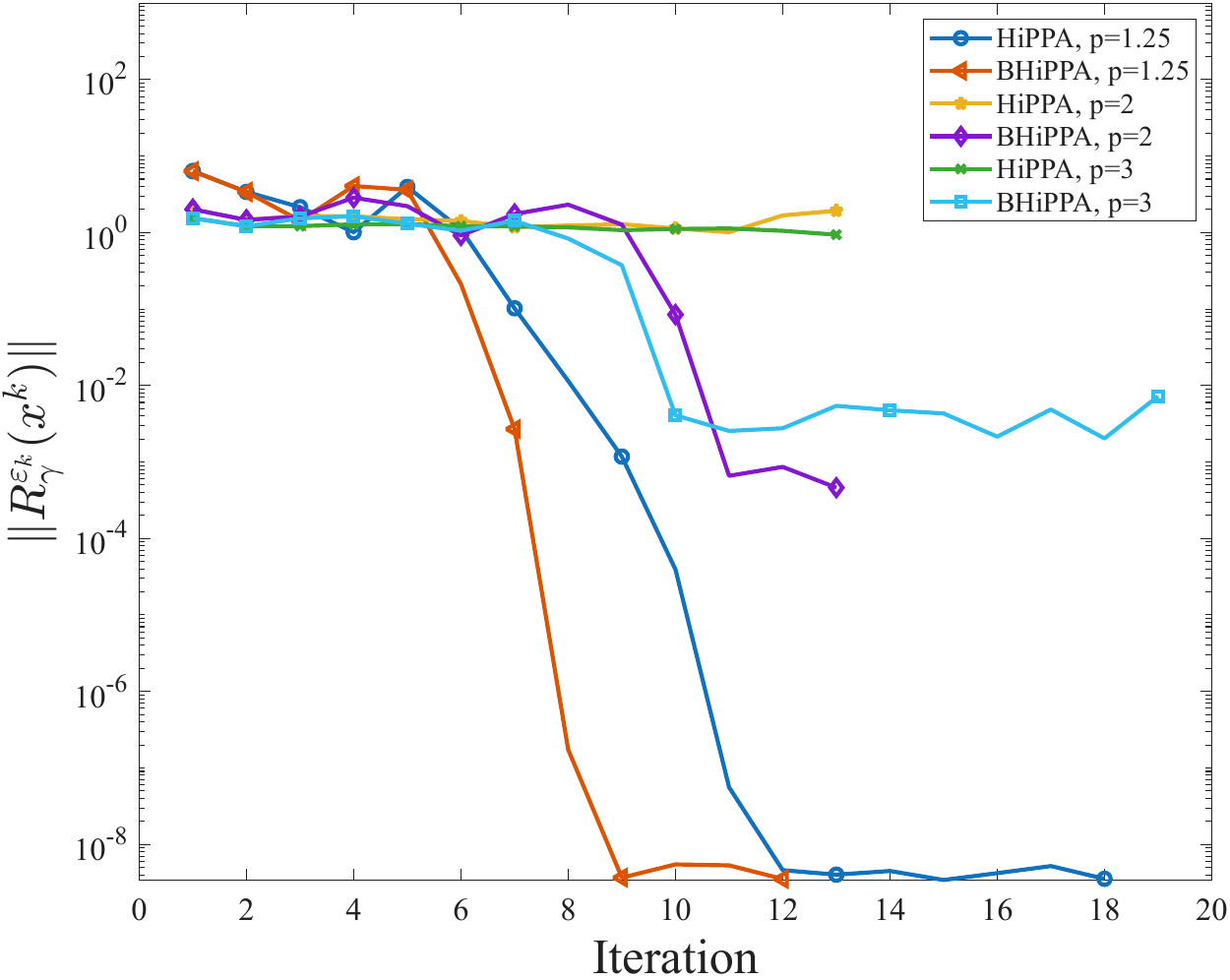}}\qquad\quad
            \subfloat[Function values vs. iterations for \eqref{eq:Experiments:optmodel}]
             {\includegraphics[width=5.9cm]{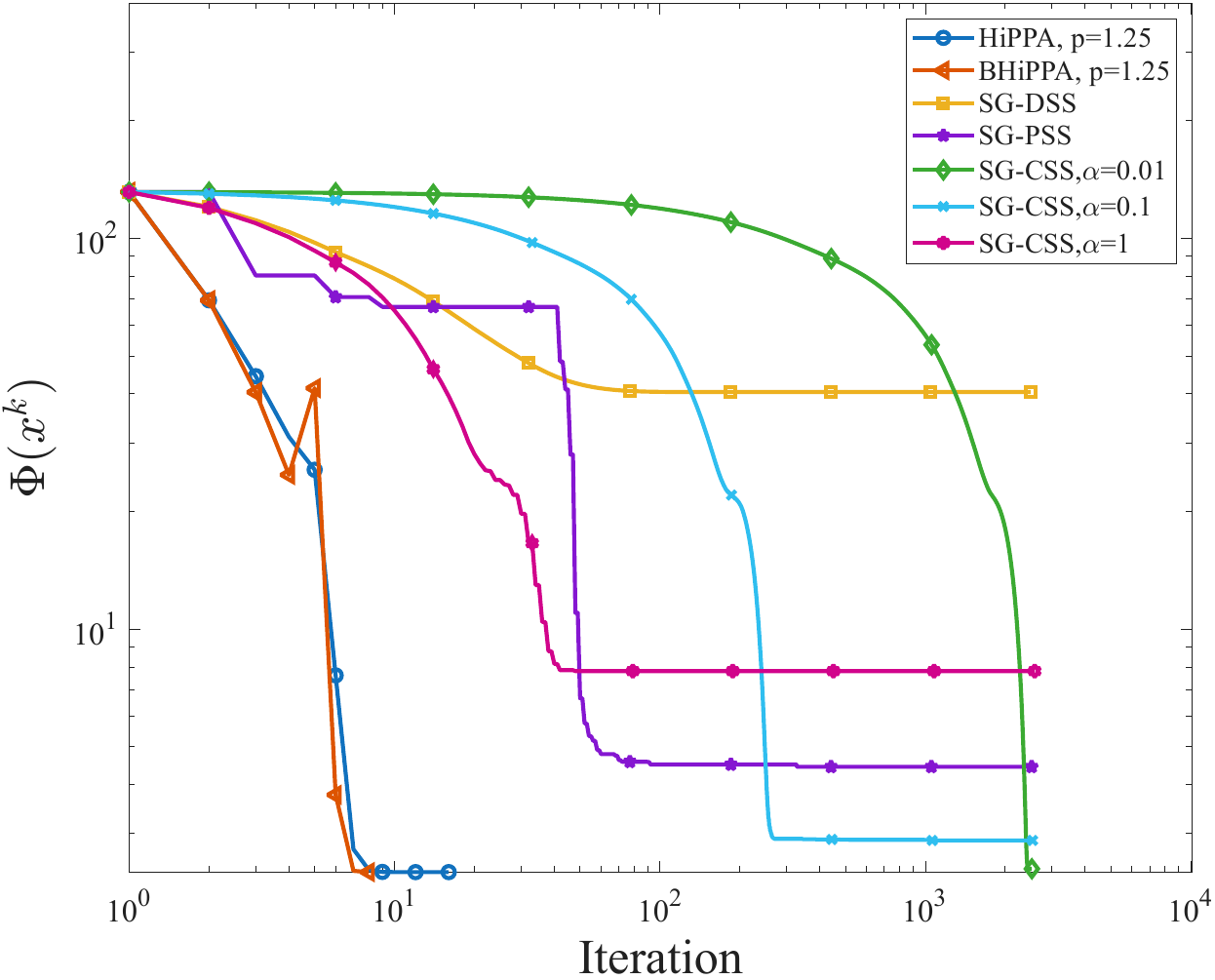}}\vspace{-2mm}\\
             \subfloat[$\Vert R_\gamma^{\varepsilon_k}(x^k)\Vert$ vs. time for \eqref{eq:Experiments:optmodel}]
            {\includegraphics[width=5.9cm]{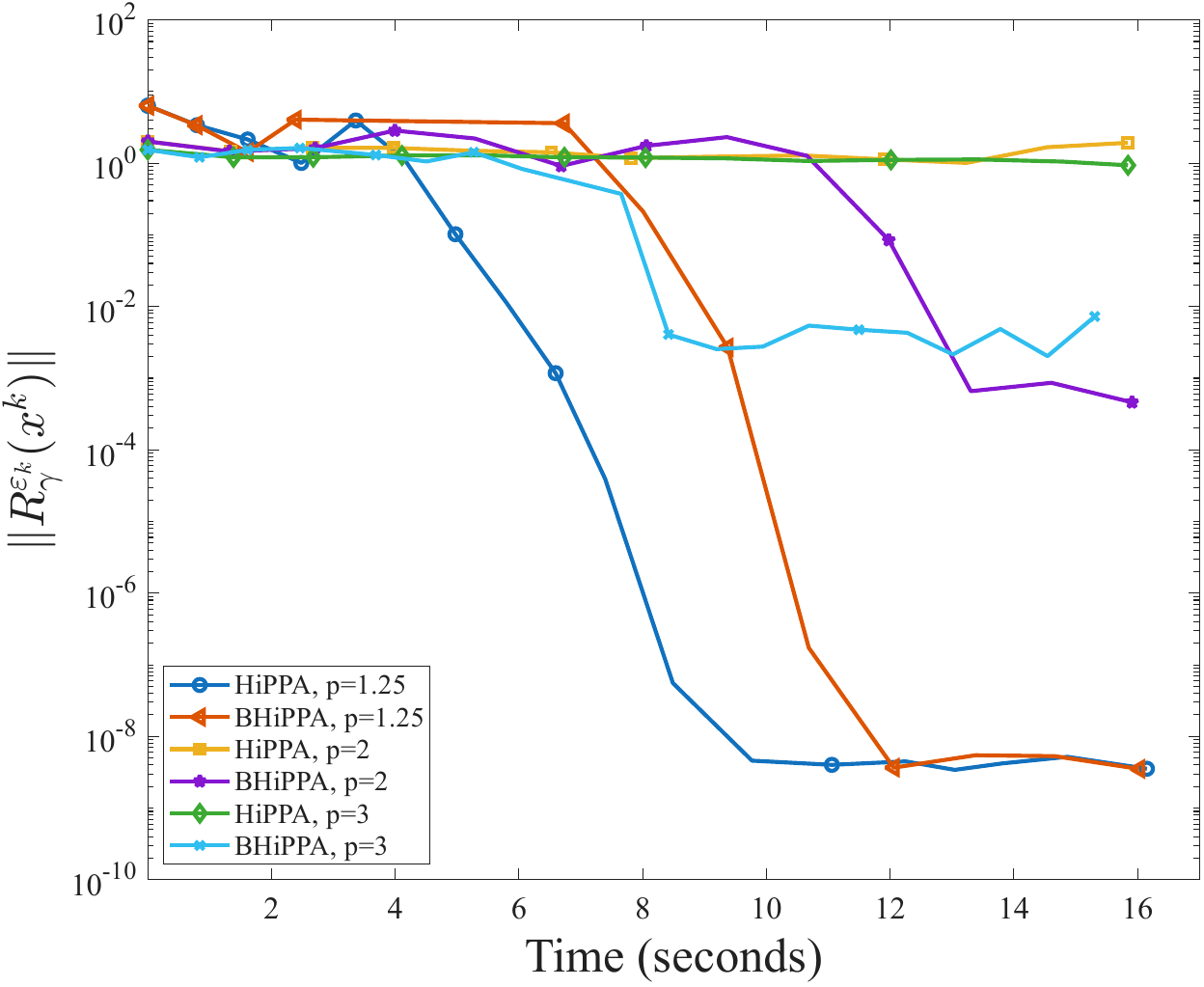}}\qquad\quad
            \subfloat[Function values vs. time for \eqref{eq:Experiments:optmodel}]
             {\includegraphics[width=5.9cm]{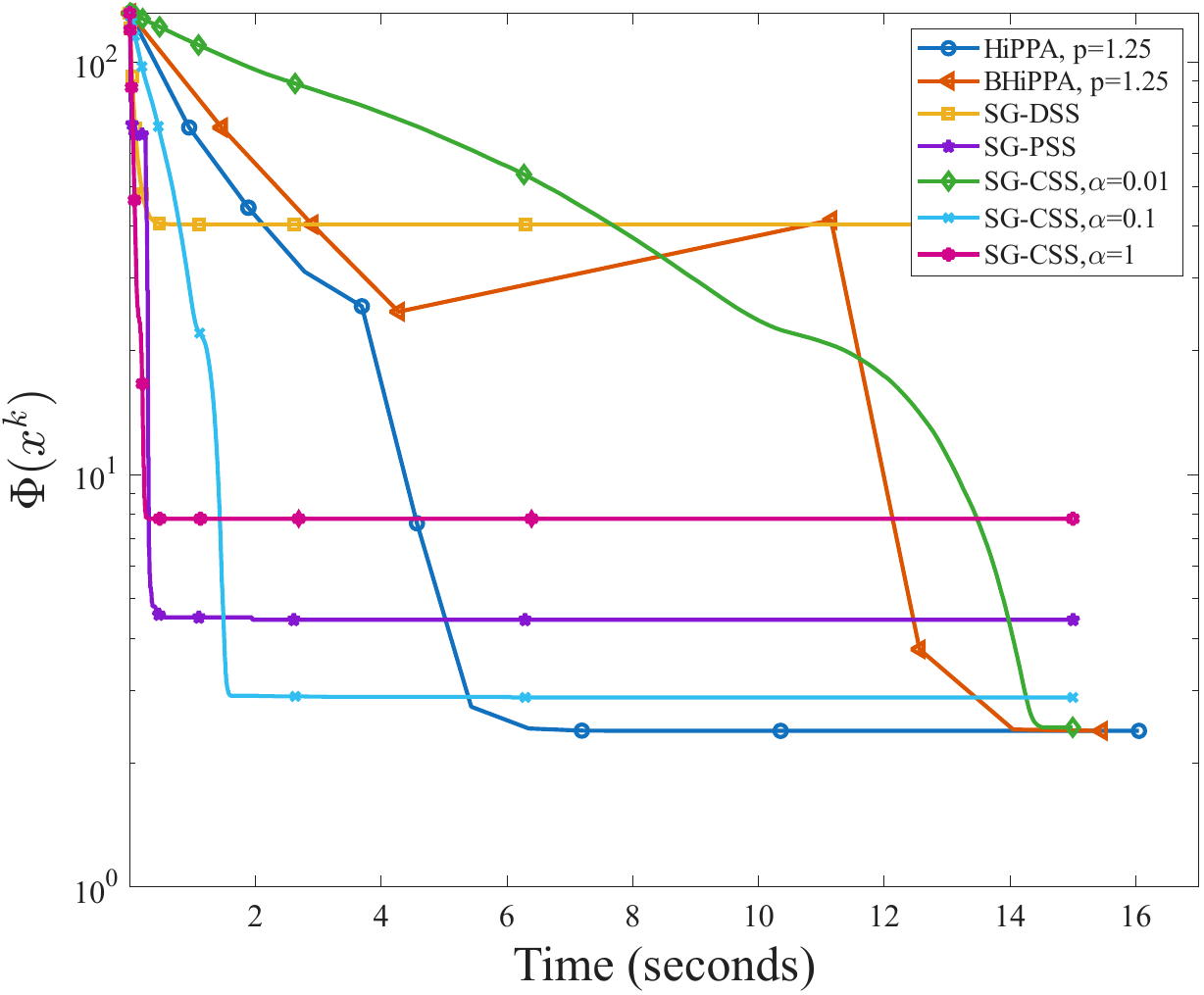}\vspace{-4mm}}        
            \caption{The sequences of residual $\left\{\Vert R_\gamma^{\varepsilon_k}(x^k)\Vert\right\}_{k\geq 0}$ and function values $\left\{\Phi(x^k)\right\}_{k\geq 0}$ versus outer iterations  and versus time for the problem \eqref{eq:Experiments:optmodel}.}
             \label{fig:ex:diffp}
    \end{center}    
\end{figure}

\begin{figure}[H]
    \begin{center}
 \subfloat[$\Vert R_\gamma^{\varepsilon_k}(x^k)\Vert$ vs. iterations for \eqref{eq:Experiments:optmodel2}]
            {\includegraphics[width=5.9cm]{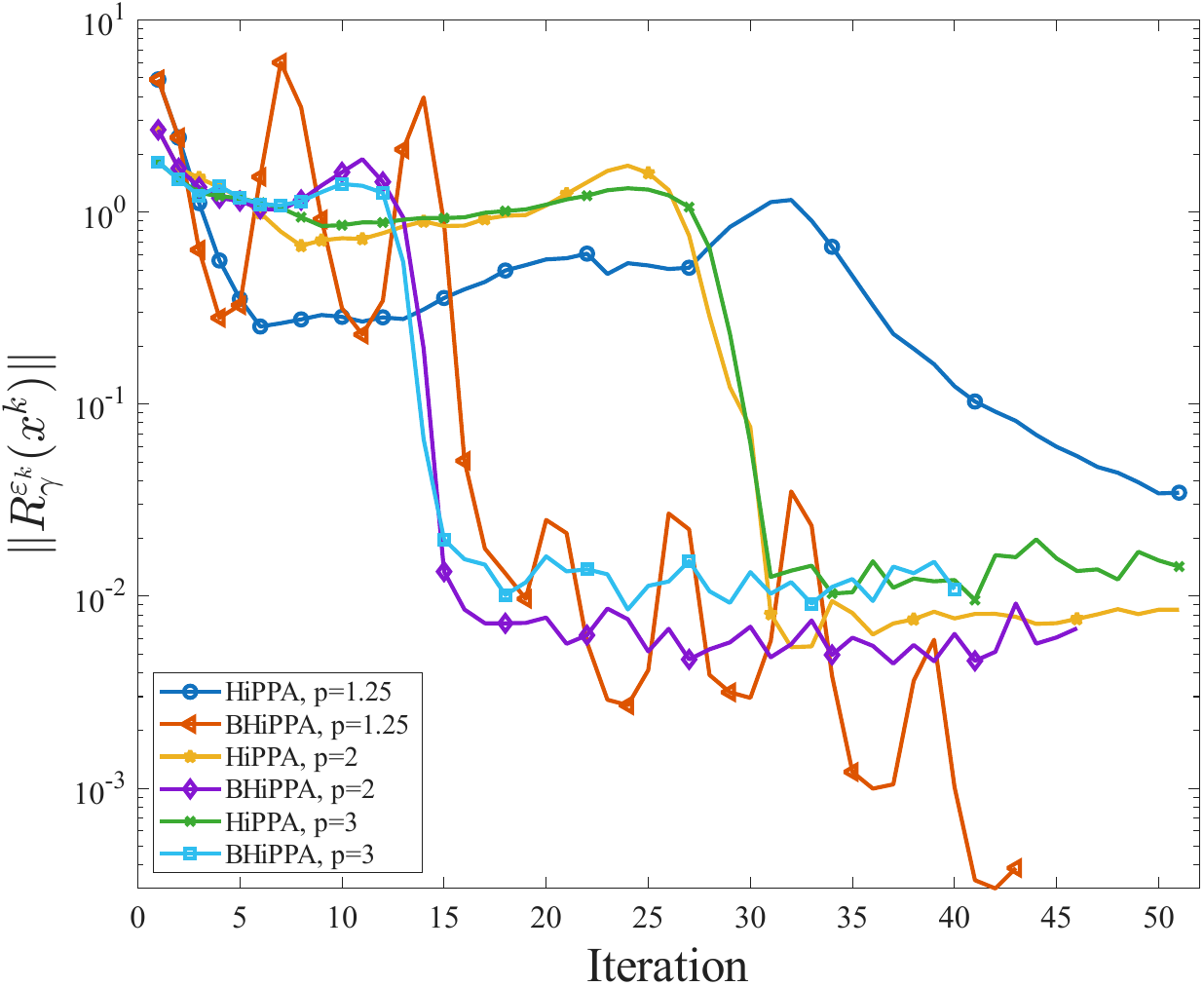}}\qquad\quad
            \subfloat[Function values vs. iterations for \eqref{eq:Experiments:optmodel2}]
             {\includegraphics[width=5.9cm]{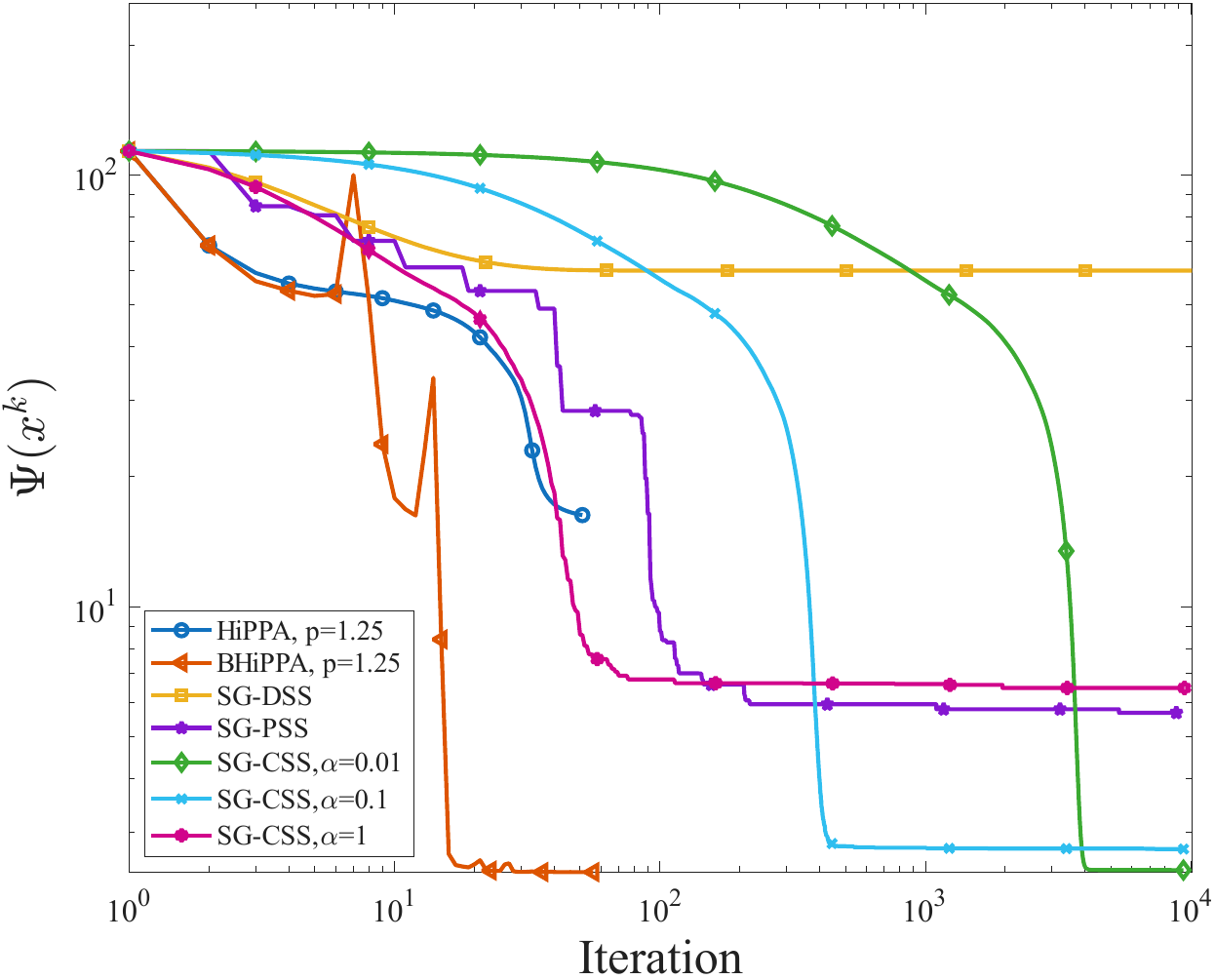}}  \\
             \subfloat[$\Vert R_\gamma^{\varepsilon_k}(x^k)\Vert$ vs. time for \eqref{eq:Experiments:optmodel2}]
            {\includegraphics[width=5.9cm]{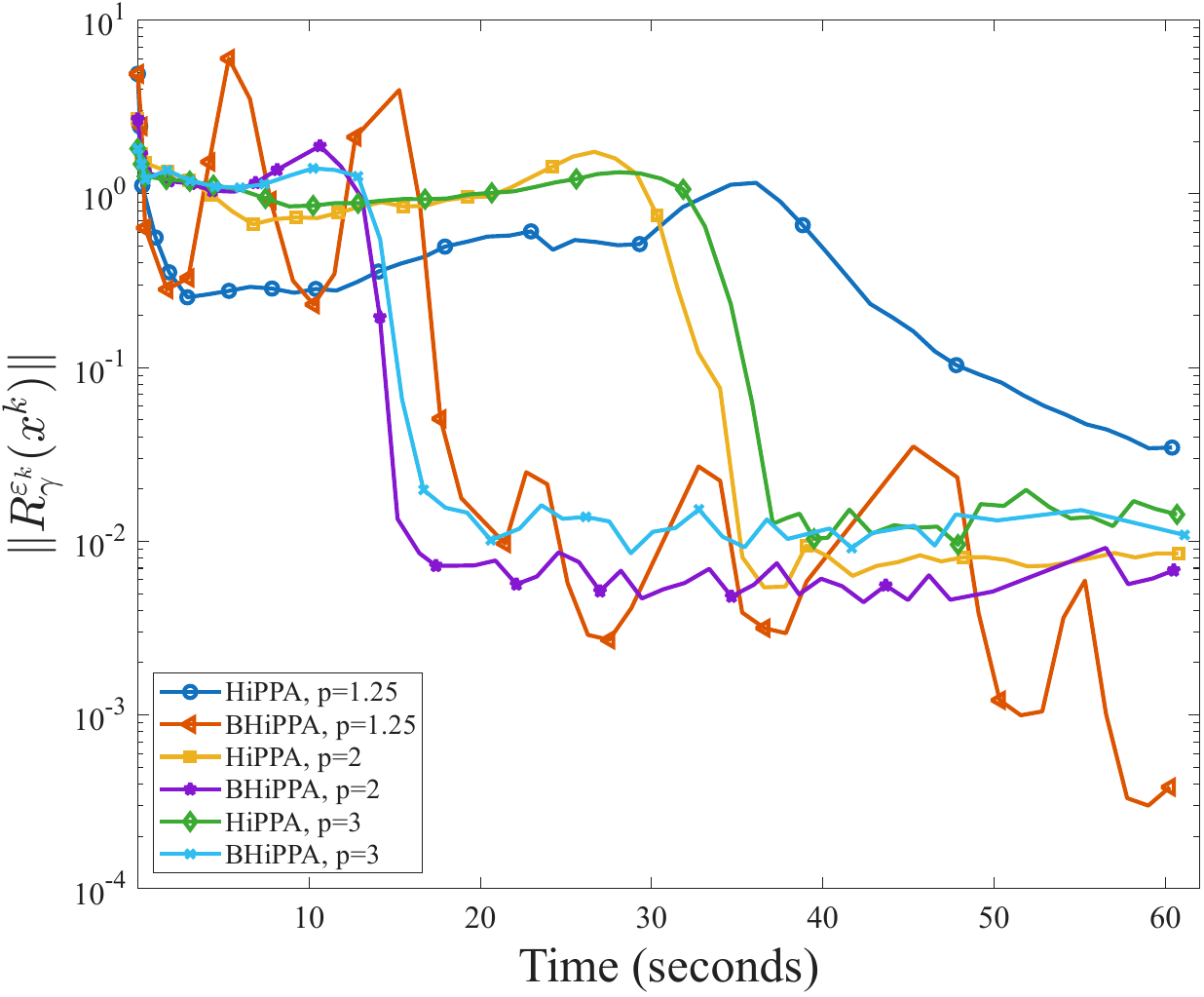}}\qquad\quad
            \subfloat[Function values vs. time for \eqref{eq:Experiments:optmodel2}]
             {\includegraphics[width=5.9cm]{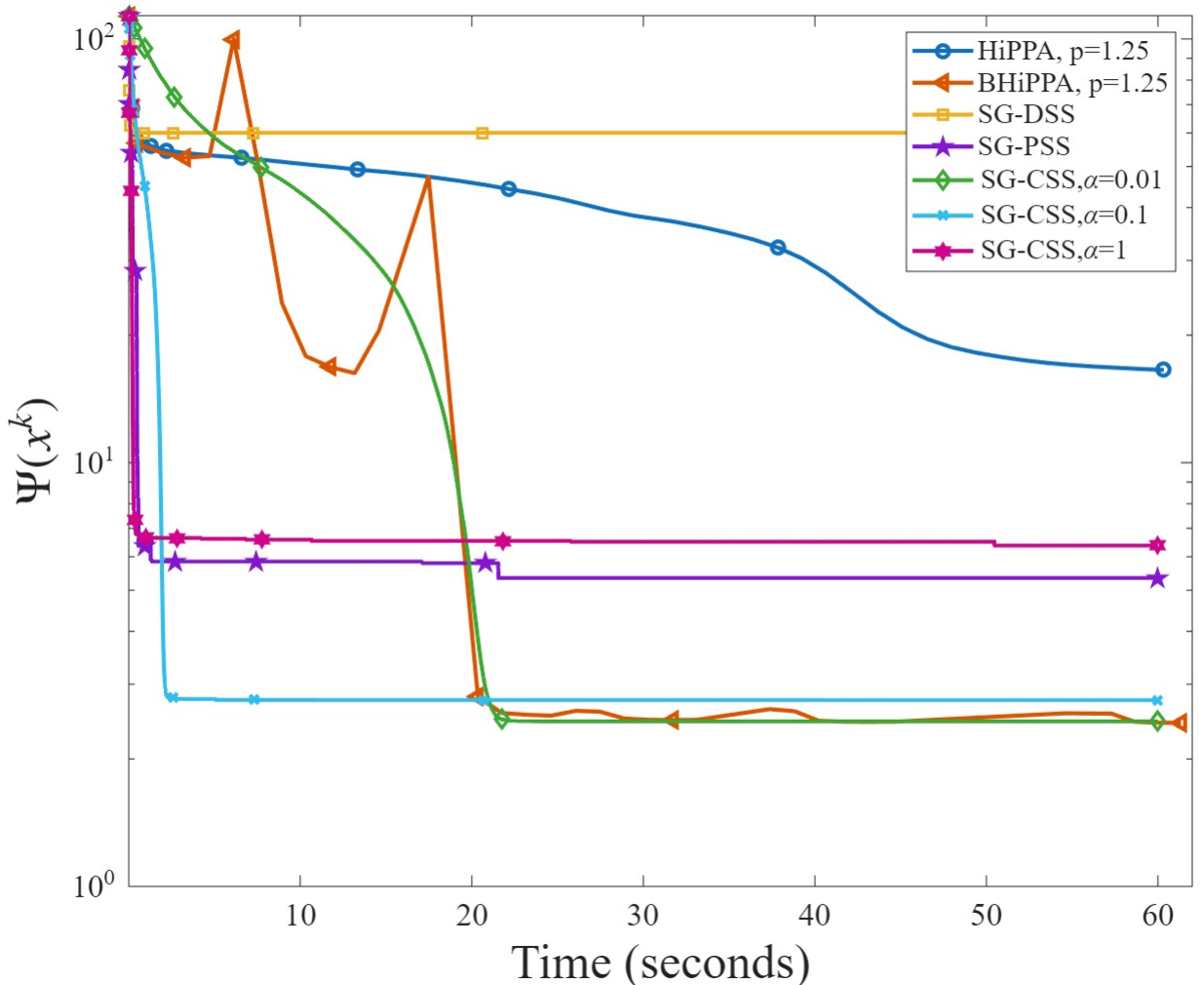}\vspace{-4mm}}
            \caption{The sequences of residual $\left\{\Vert R_\gamma^{\varepsilon_k}(x^k)\Vert\right\}_{k\geq 0}$ and function values $\left\{\Psi(x^k)\right\}_{k\geq 0}$ versus outer iterations  and versus time for the problem \eqref{eq:Experiments:optmodel2}.}
             \label{fig:ex:diffp2}
    \end{center}    
\end{figure}

\section{Conclusions} \label{sec:disc}
In this paper, we introduced an inexact two-level smoothing optimization (ItsOPT) framework to address a wide class of nonsmooth and nonconvex problems. This framework represents a crucial step toward applying proximal envelopes in situations where either the objective function or the proximity regularization is complicated. In such cases, the proximal auxiliary problem cannot be solved in closed form, and consequently, the associated envelope does not admit an exact oracle. Any algorithm within this framework consists of two methodological levels:
(i) at the lower level, the proximal auxiliary problem is solved approximately according to prescribed inexactness conditions, thereby generating an inexact oracle for the envelope;
(ii) at the upper level, zeroth-, first-, or second-order methods are applied based on the generated inexact oracle.

We focused in particular on the high-order proximal operator (HOPE) as the auxiliary problem and the high-order Moreau envelope (HOME) as the smoothing tool. We studied the key analytical properties of HOME and HOPE under proximity regularization of order $p>1$. Building on these results, we proposed the Boosted Inexact High-Order Proximal-Point Method (Boosted HiPPA), a line search acceleration of the inexact high-order proximal-point method, as the first zeroth-order natural algorithm arising from the ItsOPT framework.
We established subsequential convergence of Boosted HiPPA under reasonable assumptions on the quality of the inexact proximal step, and proved global and linear convergence under the KL property. Notably, when the original objective is a KL function with exponent $\theta \in (0, 1)$, setting  $p = \frac{1}{1-\theta}$ ensures linear convergence. To our knowledge, Boosted HiPPA is the first algorithm that achieves linear convergence for any KL function by appropriately leveraging proximal regularization. 

Finally, we note that some first-order methods in the sense of the ItsOPT framework have been developed in \cite{Kabganitechadaptive}. The development of second-order instances ItsOPT requires the second-order characterization of HOME, which constitutes an important direction for future research.




\bibliographystyle{siamplain}
\bibliography{references.bib}
\end{document}